\documentclass[11pt]{article}
\usepackage[margin=1.3in]{geometry}
\usepackage{comment}
\usepackage{amsfonts, amssymb,amsmath,amsthm}
\RequirePackage[authoryear]{natbib} 
\RequirePackage[colorlinks,citecolor=blue,urlcolor=blue]{hyperref}
\RequirePackage{graphicx}
\usepackage{txfonts}

\usepackage[utf8]{inputenc} 
\usepackage[T1]{fontenc}    
\usepackage{hyperref}       
\usepackage{url}            
\usepackage{booktabs}       
\usepackage{xcolor}
\usepackage{nicefrac}       
\usepackage{microtype}      
\usepackage{centernot}
\usepackage[boxruled,
linesnumbered,
commentsnumbered]{algorithm2e}
\usepackage[sans]{dsfont}

\usepackage{mathtools}
\usepackage{graphicx}
\usepackage{enumerate,xspace}
\usepackage{fancyhdr}
\usepackage{comment}
\usepackage{parskip}
\usepackage{float}
\usepackage{longtable}
\usepackage{tablefootnote}
\usepackage{url}
\usepackage{standalone}
\restylefloat{table}

\usepackage{tikz}
\usetikzlibrary{matrix,arrows,calc,shapes,backgrounds}
\usetikzlibrary{shapes.callouts,decorations.text}
\usetikzlibrary{shapes.misc}

\newcommand{\reals}{\mathbb{R}}

\newcommand{\Indc}{\mathbf{1}}

\newcommand{\E}{\mathbb{E}}

\newcommand{\Prob}{\mathbb{P}}
\newcommand{\Indicator}{\mathds{1}}

\newcommand{\n}[1]{%
\ifthenelse{\equal{#1}{}}{\frac{1}{n}}{\frac{1}{n^{#1}}}%
}

\newcommand{\Trans}{\mathrm{T}}
\newcommand{\diff}{\mathrm{d}}

\renewcommand{\epsilon}{\varepsilon}

\newcommand{\bX}{\mathbf{X}}

\newcommand{\ba}{\mathbf{a}}

\newcommand{\be}{\mathbf{e}}
\newcommand{\bg}{\mathbf{g}}

\newcommand{\bx}{\mathbf{x}}
\newcommand{\by}{\mathbf{y}}
\newcommand{\bz}{\mathbf{z}}
\newcommand{\bff}{\mathbf{f}}

\newcommand{\muhat}{\widehat{\mu}}

\newcommand{\bmu}{\boldsymbol{\mu}}
\newcommand{\bmuhat}{\widehat{\bmu}}

\newcommand{\bpsi}{\boldsymbol{\psi}}
\newcommand{\btheta}{\boldsymbol{\theta}}

\newcommand{\defeq}{\vcentcolon=}

\newcommand{\Y}{\mathcal{Y}}
\newcommand{\calH}{\mathcal{H}}
\newcommand{\A}{\mathcal{A}}
\newcommand{\X}{\mathcal{X}}

\newcommand{\N}{\mathcal{N}}
\renewcommand{\L}{\mathcal{L}}
\newcommand{\F}{\mathcal{F}}
\newcommand{\cP}{\mathcal{P}}
\newcommand{\G}{\mathcal{G}}
\newcommand{\I}{\mathcal{I}}
\newcommand{\R}{\mathbb{R}}

\newcommand{\norm}[1]{\left\lVert#1\right\rVert}

\DeclareMathOperator*{\argmin}{argmin}
\DeclareMathOperator*{\argmax}{argmax}

\SetCommentSty{commfont}
\SetKwComment{Comm}{$\rhd\ $}{}
\mathtoolsset{showonlyrefs,showmanualtags}

\def\ba{\mathbf{a}}

\def\bt{\mathrm{t}}
\def\bx{\mathbf{x}}

\def\bX{\mathbf{X}}
\def\bz{\mathbf{z}}

\def\bmu{\boldsymbol{\mu}}

\def\btheta{\boldsymbol{\theta}}

\DeclarePairedDelimiterX{\infdivx}[2]{(}{)}{%
  #1\;\delimsize\|\;#2%
}

\usepackage{mathrsfs}

\definecolor{linkcolor}{RGB}{0, 0, 255}
\newtheoremstyle{theoremstyle}
  {\topsep}{\topsep}{\itshape}{}{}{}{.5em}
  {\color{linkcolor}\ifthenelse{\equal{#3}{}}{{\bfseries #1 #2}}{{\bfseries #1 #2 (#3)}}\ \\}
\newtheoremstyle{examplestyle}
  {\topsep}{\topsep}{}{}{}{}{.5em}
  {\color{linkcolor}\ifthenelse{\equal{#3}{}}{{\bfseries #1 #2}}{{\bfseries #1 #2 (#3)}}}
\theoremstyle{theoremstyle}\newtheorem{theorem}{Theorem}
\theoremstyle{theoremstyle}
\newtheorem{corollary}[theorem]{Corollary}
\theoremstyle{theoremstyle}\newtheorem{lemma}{Lemma}
\theoremstyle{theoremstyle}
\theoremstyle{theoremstyle}
\theoremstyle{theoremstyle}\newtheorem{assumption}{Assumption}
\theoremstyle{theoremstyle}

\theoremstyle{theoremstyle}

\theoremstyle{theoremstyle}

\theoremstyle{examplestyle}
\theoremstyle{examplestyle}\newtheorem{remark}{Remark}
\theoremstyle{examplestyle}
\usepackage{stmaryrd}

\usepackage{comment}
\usepackage{subcaption}
\begin{document}

\title{Convergence Rates of Oblique Regression Trees for Flexible Function
  Libraries\thanks{
    The authors would like to thank
    Florentina Bunea, Sameer Deshpande, Jianqing Fan, Yingying Fan,
    Jonathan Siegel, Bartolomeo Stellato, and William Underwood
    for insightful discussions. The authors are particularly grateful to two
    anonymous reviewers whose comments improved the quality of the paper.
    MDC was supported in part by the National Science Foundation through
    SES-2019432 and SES-2241575.
    JMK was supported in part by the National Science Foundation through CAREER
    DMS-2239448, DMS-2054808, and HDR TRIPODS CCF-1934924.
}}
\author{Matias D. Cattaneo\thanks{Department of Operations Research and Financial Engineering, Princeton University.} \and
  Rajita Chandak\footnotemark[2]  \and
  Jason M. Klusowski\footnotemark[2]
}
\maketitle

\begin{abstract}
  We develop a theoretical framework for the analysis of oblique decision trees,
  where the splits at each decision node occur at linear combinations of
  the covariates (as opposed to conventional tree constructions that
  force axis-aligned splits involving only a single
  covariate). While this methodology has garnered significant attention from
  the computer science and optimization communities since the mid-80s,
  the advantages they offer over their axis-aligned
  counterparts remain only empirically justified, and explanations for their
  success are largely based on heuristics. Filling this long-standing gap
  between theory and practice, we show that oblique regression trees
  (constructed by recursively minimizing squared error) satisfy a type of oracle
  inequality and can adapt to a rich library of regression models consisting of
  linear combinations of ridge functions and their limit points. This provides a
  quantitative baseline to compare and contrast decision trees with other less
  interpretable methods, such as projection pursuit regression and neural
  networks, which target similar model forms. Contrary to popular belief, one
  need not always trade-off interpretability with accuracy. Specifically, we show
  that, under suitable conditions, oblique decision trees achieve similar
  predictive accuracy as neural networks for the same library of regression models.
  To address the combinatorial complexity of finding the optimal
  splitting hyperplane at each decision node, our proposed theoretical framework
  can accommodate many existing computational tools in the literature.
  Our results rely on (arguably surprising) connections
  between recursive adaptive partitioning and sequential greedy approximation algorithms
  for convex optimization problems (e.g., orthogonal greedy algorithms),
  which may be of independent theoretical interest. Using our theory and methods, we also study oblique random forests.
\end{abstract}

\sloppy
\section{Introduction}
Decision trees and neural networks are conventionally seen as two contrasting
approaches to learning.
The popular belief is that decision trees compromise accuracy for
being easy to use and understand,
whereas neural networks are more accurate, but at the cost of being less
transparent.
We challenge the \textit{status quo} by showing that, under suitable conditions,
oblique decision trees (also known as multivariate decision trees)
achieve similar predictive accuracy as
neural networks on the same library of regression models.
Of course, while it is somewhat subjective as to what one regards as being
transparent, it is generally agreed upon that
neural networks are less interpretable than decision trees
\citep{murdoch2019definitions,rudin2019interpretable}.
Indeed, trees are arguably more intuitive in their construction,
which makes it easier to understand how an output is assigned to a given input,
including which predictor variables were relevant in its determination.
For example, in clinical, legal, or business contexts,
it may be desirable to build a predictive model that mimics the way a
human user thinks and reasons, especially if the results
(of scientific or evidential value)
are to be communicated to a statistical lay audience.
Even though it may be sensible to deploy estimators that more directly target
the functional form of the model,
predictive accuracy is not the only factor the
modern researcher must consider when designing and building an automated
system.
Facilitating human-machine interaction and engagement  is
also an essential part of this process. To this end, the technique of knowledge distillation
\citep{Buciluundefined2006} is a quick and easy way to enhance the fidelity of
an interpretable model, without degrading the out-of-sample performance too
severely. In the context of decision trees and neural networks, one distills the
knowledge acquired by a neural network—which relies on nontransparent,
distributed hierarchical representations of the data—and expresses similar
knowledge in a decision tree that consists of, in contrast, easier to understand
hierarchical decision rules \citep{frosst2017distilling}.
This is accomplished
by first training a neural network on the observed data, and then, in
turn, training a decision tree on data generated from the fitted neural network model.

In this paper, we show that oblique regression trees
(constructed by recursively minimizing squared error)
satisfy a type of oracle inequality and can adapt to a rich library
of regression models consisting of linear combinations of ridge functions.
This provides a quantitative baseline to compare and
contrast decision trees with other less interpretable methods, such as
projection pursuit regression, neural networks, and boosting machines, which
directly target similar model forms. When neural network and decision tree
models are used in tandem to enhance generalization and interpretability,
our theory allows one to measure the knowledge distilled from a neural network
to a decision tree. Using our theory and methods, we also study oblique random forests.

\subsection{Background and Prior Work}

Let
$(y_1, \bx_1^\Trans), \ldots, (y_n, \bx_n^\Trans)$
be a random sample from a joint distribution
$ \Prob_{(y, \bx)} = \Prob_{y \mid \bx}\Prob_{\bx}  $
supported on
$\Y \times \X$.
Here
$ \bx = (x_1, \dots, x_p)^{\Trans}$
is a vector of
$ p $ predictor variables supported on
$ \X \subseteq \reals^p $ and
$ y $ is a real-valued outcome variable with range $\Y \subseteq \reals$.
Our objective is to compute an estimate of the
conditional expectation, $\mu(\bx) = \E[y \mid \bx]$,
a target which is optimal for predicting $ y $ from some function of $ \bx $ in
mean squared error.
One estimation scheme can be constructed by dividing the input space $ \X $
into subgroups based on shared characteristics of $ y $—something decision trees
can do well.

A decision tree is a hierarchically organized data structure
constructed in a top down, greedy manner through recursive binary splitting.
According to CART methodology \citep{breiman1984cart}, a
parent node $ \bt $ (i.e., a region in $\X$) in the tree is divided into two
child nodes, $ \bt_L $ and $ \bt_R $, by maximizing the decrease in
sum-of-squares error (SSE)
\begin{equation}
  \label{eq:sse}
  \widehat\Delta(b, \ba, \bt)
  =
  \frac{1}{n} \sum_{\bx_i \in \bt}(y_i - \overline y_{\bt})^2
  -
  \frac{1}{n}
  \sum_{\bx_i \in \bt}
  (y_{i} - \overline y_{\bt_L} \Indicator(\ba^\Trans\bx_i \leq b)
  - \overline y_{\bt_R} \Indicator(\ba^\Trans \bx_i > b))^2,
\end{equation}
with respect to $(b, \ba)$, with $\Indicator(\cdot)$ denoting the indicator
function and $ \overline y_{\bt} $ denoting the sample average of the $y_i$ data
whose corresponding $\bx_i$ data lies in the node $ \bt $. In the conventional
\emph{axis-aligned} (or, \textit{univariate}) \textit{CART} algorithm
\citep[Section 2.2]{breiman1984cart},
splits occur along values of a single covariate, and so the search space for $\ba$ is
restricted to the set of standard basis vectors in $ \reals^p $. In this case,
the induced partition of the input space $\X$ is a set of hyper-rectangles. On
the other hand, the \emph{oblique CART} algorithm
\citep[Section 5.2]{breiman1984cart}
allows for linear combinations of covariates, extending
the search space for $\ba$ to be all of $ \reals^p $. Such a procedure generates
regions in $\reals^p$ that are convex polytopes.

The solution of~\eqref{eq:sse} yields estimates $(\hat b, \hat \ba)$, and the
refinement of $\bt$ produces child nodes $ \bt_L = \{\bx \in \bt : \hat \ba^\Trans
\bx \leq \hat b\} $ and $ \bt_R = \{\bx \in \bt : \hat \ba^\Trans \bx > \hat b\} $.
These child nodes become new parent nodes at the next level of the tree and can
be further refined in the same manner until a desired depth is reached.
To obtain a maximal decision tree $T_K$ of depth $K$, the
procedure is iterated $K$ times or until either (i) the node contains a single data
point $(y_i, \bx^\Trans_i)$ or (ii) all input values $ \bx_i $ and/or all response
values $y_i$ within the node are the same.
The maximal decision tree with maximum depth is denoted by $T_{\text{max}}$.
An illustration of a maximal oblique
decision tree with depth $K = 2$ is shown in Figure~\ref{fig:tree}. For
contrast, in Figure~\ref{fig:aligned_tree}, we show a maximal axis-aligned
decision tree with depth $K = 2$.

In a conventional regression problem, where the goal is to estimate the
conditional mean response $ \mu(\bx) $,
the canonical tree output for $ \bx \in \bt $ is $ \overline y_{\bt} $,
i.e., if $ T $ is a decision tree,
then
\begin{equation}\label{eq:output}
  \muhat(T)(\bx)
= \overline y_{\bt}
=
\frac{1}{n(\bt)}\sum_{\bx_i \in \bt}y_i,
\end{equation}
where $n(\bt)$ denotes the number of observations in the node $\bt$.
However,
one can aggregate the data in each node in a number of ways,
depending on the form of the target estimand.
In the most general setting, under weak assumptions,
all of our forthcoming theory holds when the node
output is the result of a least squares projection onto the linear span of a
finite dictionary
$ \mathcal{H} $ that includes the constant function
(e.g., polynomials, splines),
that is,
$ \hat y_{\bt}
\in \argmin_{h \in\text{span}(\mathcal{H})}
\sum_{\bx_i \in \bt}(y_i - h(\bx_i))^2 $.

\begin{figure}[ht]
  \begin{subfigure}{0.4\textwidth}
  \centering
  \includegraphics[width=0.55\textwidth]{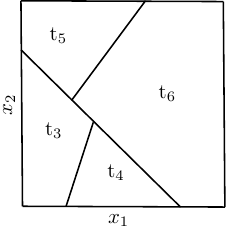}
  \end{subfigure}
  \hspace{0.4cm}
  \begin{subfigure}{0.4\textwidth}
    \centering
     \vspace{-0.4cm}
    \includegraphics[width=1.1\textwidth]{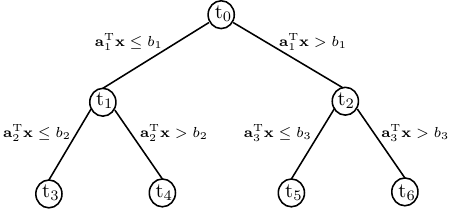}
  \end{subfigure}
  \caption{A maximal oblique decision tree with depth $ K = 2 $ in $ p = 2$
    dimensions.
    Splits occur along hyperplanes of the form $ a_1x_1 + a_2x_2 = b $.}~\label{fig:tree}
\end{figure}

\begin{figure}[ht]
  \begin{subfigure}{0.4\textwidth}
    \centering
    \includegraphics[width=0.55\textwidth]{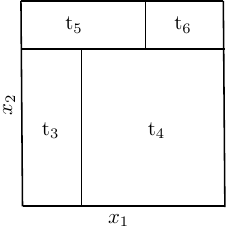}
  \end{subfigure}
  \hspace{0.4cm}
  \begin{subfigure}{0.4\textwidth}
    \centering
    \vspace{-0.4cm}
    \includegraphics[width=1.1\textwidth]{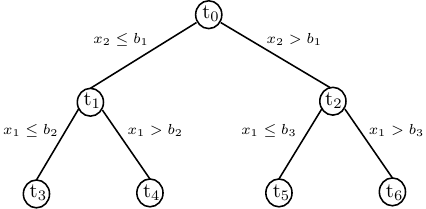}
  \end{subfigure}
  \caption{A maximal axis-aligned decision tree with depth $ K = 2 $ in $ p = 2$
    dimensions.
    Splits occur along individual covariates of the form $ x_j = b$ for $ j = 1, 2$.}~\label{fig:aligned_tree}
\end{figure}
\begin{figure}[ht]
  \centering
  \includegraphics[width=0.45\textwidth]{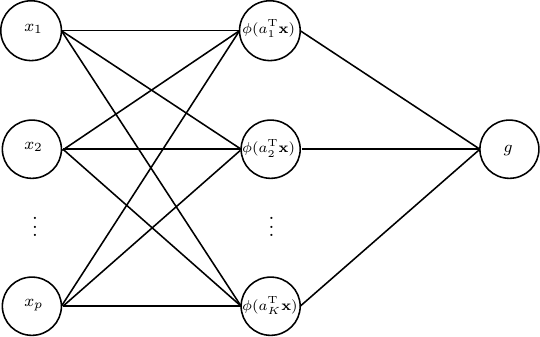}
  \caption{A single hidden layer neural network with $K$ hidden nodes.}
  \label{fig:nn_tree}
\end{figure}
One of the main practical issues with oblique CART is that the computational
complexity of minimizing the squared error in~\eqref{eq:sse} in each node is
extremely demanding (in fact, it is NP-hard). For example, if we desire to split
a node $ \bt $ with $ n(\bt) $ observations for axis-aligned CART,
an exhaustive search would require at most $ p \cdot n(\bt) $ evaluations,
whereas oblique CART would require a prodigious $2^p \binom{n(t)}{p}$ evaluations
\citep{murthy1994system}.

To deal with these computational demands, \citet{breiman1984cart} first suggested
a method for inducing oblique decision trees. They use a fully deterministic
hill-climbing algorithm to search for the best oblique split. A backward feature
elimination process is also carried out to delete irrelevant features from the
split. \citet{Heath93inductionof} propose a simulated annealing optimization
algorithm, which uses randomization to search for the best split to
potentially avoid getting stuck in a local optimum.
\citet{murthy1994system} use a combination of deterministic hill-climbing and
random perturbations in an attempt to find a good hyperplane.
See \citet{brodley1995multivariate} for additional variations on these
algorithms.
Other works employ statistical techniques like
linear discriminant analysis (LDA)
\citep{lopez2013fisher, li2003multivariate, loh1997split},
principle components analysis (PCA)
\citep{menze2011oblique, rodriguez2006rotation},
and random projections
\citep{tomita2020sparse}.

While not the focus of the present paper, regarding non-greedy training, other
researchers have attempted to find globally optimal tree solutions using linear
programming \citep{bennett1994global} or mixed-integer linear programming
\citep{bertsimas2017optimal, bertsimas2021near}. It should be clear that all of
our results hold verbatim for optimal trees, as greedy implementations belong to
the same feasible set. While usually better than greedy trees in terms
of predictive performance, scalability to large data sets is the most salient
obstacle with globally optimal trees. Moreover, on a qualitative level, a globally
optimal tree arguably detracts from the interpretability, as humans, in
contrast, often exhibit bounded rationality and therefore make decisions in a
more sequential (rather than anticipatory) manner
\citep[and references therein]{hullermeier2021automated}.
Relatedly, another training technique is based on
constructing deep neural networks that realize oblique decision trees
\citep{Lee2020Oblique, yang2018deep} and then utilizing tools designed for
training neural networks.

While there has been a plethora of greedy algorithms over
the past 30 years for training oblique decision trees, the literature is
essentially silent on their statistical properties. For instance, assuming one
can come close to optimizing~\eqref{eq:sse}, what types of regression
functions can greedy oblique trees estimate and how well?

\subsection{Ridge Expansions}
~\label{sec:fclass}
Many empirical studies reveal that oblique trees generally produce smaller trees
with better accuracy compared to axis-aligned trees
\citep{Heath93inductionof, murthy1994system}
and can often be comparable, in terms of performance, to neural networks
\citep{bertsimas2018optimal, bertsimas2019machine, bertsimas2021voice}.
Intuitively, allowing a tree-building system to use
both oblique and axis-aligned splits broadens its flexibility. To theoretically
showcase these qualities and make comparisons with other procedures (such as
neural networks and projection pursuit regression),
we will consider modeling $\mu$ with finite linear combinations of ridge functions,
i.e., the library
\begin{equation}
  \label{eq:fclass}
  \G = \Bigg\{
  g(\bx) = \sum_{k=1}^M g_{k}(\ba_k^{\Trans}\bx),\; \ba_k \in \R^p, \;
  g_k: \R \mapsto \R, \; k=1,\dots,M, \; M \geq 1, \; \|g\|_{\L_1} < \infty \Bigg\},
\end{equation}
where $\|\cdot\|_{\L_1}$ is a total variation norm that is defined in Section~\ref{sec:notation}.
This library encompasses the functions produced from projection
pursuit regression,
and, more specifically—by taking $ g_k(z) = \phi(z-b_k) $, where $ \phi $ is a
fixed activation function, such as a sigmoid function or ReLU, and $ b_k \in
\mathbb{R} $ is a bias parameter—single hidden layer feed-forward neural
networks. A graphical representation of such a neural network is provided in
Figure~\ref{fig:nn_tree}. A neural network forms predictions according to
distributed hierarchical representations of the data, whereas a decision tree
uses hierarchical decision rules (c.f., Figures~\ref{fig:tree}
and~\ref{fig:aligned_tree}).

Since the first version of our manuscript was released on \texttt{arXiv},
several subsequent papers have employed our novel theoretical and methodological
statistical framework to derive consistency results for decision tree and related methods.
For example, \cite{zhan2023consistency} applies our core ideas and proof
techniques to deduce a consistency result for oblique decision trees in
low-dimensional settings (c.f., Corollary \ref{cor:fixed_d} below), but under
stronger assumptions on the target function class $\G$ and without accounting
for the underlying optimization constraints (c.f., our novel optimization
framework in Section \ref{sec:optim framework}). \cite{raymaekers2023fast} also
applies our core ideas and proof techniques to deduce a consistency result for
axis-aligned decision trees within an alternative computation framework, but
under stronger assumptions on the target function class $\G$. Finally,
\cite{parhi2023deep} and \cite{devore2023weighted}, among others (see their
references), study consistency of deep neural network methods using similar
notions of Hilbert function spaces and total variation norms as our paper does
for adaptive decision trees and shallow neural networks, but without accounting
for the underlying optimization constraints. In particular, \cite{parhi2023deep}
also show that neural networks are able to adapt to sparsity in the data (c.f.,
Section \ref{sec:fast} below).

\section{Main Results} \label{sec:main}
We first introduce notation and assumptions that are used
throughout the remainder of the paper.

\subsection{Notation and Assumptions}
\label{sec:notation}
For a function $ f:\mathbb{R}^p \to \mathbb{R} $, we define
$ \mathbf{f} = (f(\bx_1), \dots, f(\bx_n))^{\Trans} $
to be the $ n\times 1$ vector of $ f $
evaluated at the design points $\bX = (\bx_1, \dots, \bx_n)^\Trans \in
\mathbb{R}^{n\times p} $. Likewise, we use $\bmuhat(T_K)$ to denote the
$n \times 1$ vector of fitted values of $ \widehat\mu(T_K) $.
For functions
$ f,g \in \mathscr{L}_2(\mathbb{P}_{\bx}) $,
let $ \|f\|^2 = \int_{\X} (f(\bx))^2d\mathbb{P}_{\bx}(\bx) $
be the squared $ \mathscr{L}_2(\mathbb{P}_{\bx}) $ norm
and let
$ \|\bff\|^2_n = \frac{1}{n} \sum_{i=1}^n (f(\bx_i))^2 $
denote the squared norm
with respect to the empirical measure on the data.
Let
$\langle \bff, \bg \rangle_n = \frac{1}{n}\sum_{i=1}^n f(\bx_i)g(\bx_i) $
denote the inner product
with respect to the empirical measure on the data.
The response data vector
$ \by = (y_1, \dots, y_n)^\Trans \in \mathbb{R}^n$
is viewed as a relation, defined on the design matrix
$\bX $,
that associates $ \bx_i $ with $ y_i $.
Thus, for example,
$ \|\by-\bff\|^2_n = \frac{1}{n}\sum_{i=1}^n (y_i-f(\bx_i))^2 $
and
$ \langle \by, \bff \rangle_n = \frac{1}{n}\sum_{i=1}^n y_i f(\bx_i) $.
We use
$[T]$
to denote the collection of internal (non-terminal)
nodes and
$\{\bt: \bt \in T\}$
to denote the terminal nodes of the tree.
The cardinality of a set $A$ is denoted by $ |A| $.

We define the total variation of a ridge function
$\bx \mapsto h(\ba^\Trans \bx)$
with $\ba \in \R^p $
and $ h: \mathbb{R} \to \mathbb{R} $
in the node $ \bt $ as
\[
V(h, \ba, \bt) = \sup_\cP \sum_{\ell=0}^{|\cP|-1}|h(z_{\ell+1})-h(z_{\ell})|,
\]
where the supremum is over all partitions
$ \cP = \{z_0, z_1, \dots, z_{|\cP|} \} $
of the interval
$ I(\ba, \bt) = [\min_{\bx \in \bt}\ba^\Trans\bx,\; \max_{\bx \in
  \bt}\ba^\Trans\bx] \subset \R $
(we allow for the possibility that one or both of the endpoints is infinite).
If the function $ h $ is smooth,
then $ V(h, \ba, \bt) $ admits the familiar integral representation
$ \int_{I(\ba, \bt)}|h^{\prime}(z)|\diff z $.
We can then define the $\L_1$ norm of an additive function
$h(\bx) = \sum_{k=1}^M h_k(\bx)$ as
\begin{equation*}
  \|h\|_{\L_1} = \sum_{k=1}^M V(h_k, \ba_k, \bt).
\end{equation*}
Central to our results is the
$\L_1 $ total variation norm of
$f \in \F = \text{cl}(\G)$ in the node $ \bt $, the closure being taken in
$\mathscr{L}_2(\Prob_{\bx})$.
This quantity captures the local capacity of a function in
$\F $. It is defined as
\begin{equation*}
  \|f\|_{\L_1(\bt)}
  \defeq
  \lim_{\epsilon \downarrow 0}\inf_{g\in \G}\Bigg\{
  \sum_{k=1}^M V(g_k, \ba_k, \bt) :
  g(\bx) = \sum_{k=1}^M g_k(\ba^{\Trans}_k\bx),
  \; \|f-g\| \leq \epsilon \Bigg\}.
\end{equation*}
For simplicity, we write $ \|f\|_{\L_1} $ for $\|f\|_{\L_1(\X)}$.
This norm may be thought of as an $\ell_1$ norm on the coefficients in a
representation of the function $f$
by elements of a normalized dictionary of ridge functions.
A classic result of \citet{barron1993universal} shows that, for any function $ f $
defined on $\X = [0,1]^p$, we have the bound
$\|f\|_{\L_1} \lesssim \int \|\btheta\|_{\ell_1} |\tilde f(\btheta)|\diff\btheta $,
where $\tilde f $ is the Fourier transform of $
f$ and $\|\cdot\|_{\ell_1} $ is the usual $ \ell_1 $
norm of a vector in $\R^p$. Furthermore, there exists an $M$-term linear
combination of sigmoidal ridge functions in $ \G $ whose $
\mathscr{L}_2(\mathbb{P}_{\bx}) $ distance from $ f $ is $
O\big(\|f\|_{\L_1}/\sqrt{M}\big) $.

\subsection{Computational Framework}\label{sec:optim framework}
As mentioned earlier, it is challenging to find the direction $\hat{\ba}$ that
optimizes $ \widehat\Delta(b, \ba, \bt) $. Many of the aforementioned
computational papers address the problem by restricting the search space to a
more tractable subset of candidate directions
$\A_{\bt}$
with sparsity
$$ \sup\{\|\ba\|_{\ell_0}: \ba \in \A_{\bt}\} \leq d,$$
for some positive integer $ d $,
where $\|\ba\|_{\ell_0}$ counts the number of nonzero coordinates of $\ba$.
Because such search
strategies are sometimes unlikely to find the global maximum, we
theoretically measure their success by specifying a sub-optimality (slackness)
parameter $\kappa \in (0, 1]$ and considering the probability
$ P_{\A_{\bt}}(\kappa) $
that the maximum of
$ \widehat\Delta(b, \ba, \bt) $
over
$ \ba \in \A_{\bt} \subseteq \reals^p $
is within
a factor $\kappa $ of the maximum of $ \widehat\Delta(b, \ba, \bt) $ on the
unrestricted parameter space, $ \ba \in \R^p $.
That is, to theoretically quantify the sub-optimality of the chosen hyperplane,
we measure
\begin{equation} \label{eq:split_prob}
  P_{\A_{\bt}}(\kappa)
  =
  \Prob_{\A_{\bt}} \Bigg(
  \max_{(b, \ba) \in \mathbb{R}\times \A_{\bt}}\widehat\Delta(b, \ba, \bt)
  \geq
  \kappa \max_{(b, \ba) \in \mathbb{R}^{1+p}}\widehat\Delta(b, \ba, \bt)
  \Bigg),
\end{equation}
where $ \Prob_{\A_{\bt}} $ denotes the probability with respect to
the randomness in the search spaces $ \A_{\bt}$, conditional on the data.
The maximum of
$ \widehat \Delta(b, \ba, \bt) $ over $ (b, \ba) $
is achieved because the number of distinct values of
$ \widehat \Delta(b, \ba, \bt) $ is finite
(at most the number of ways of dividing $ n $ observations into two groups,
or, $2^n-1$).

Another way of thinking about $ P_{\A_{\bt}}(\kappa) $
is that it represents the degree of optimization misspecificity of $ \A_{\bt}$
for the form of the global optimum $ \hat \ba $.
For example, if
$ \A_{\bt}= \{\mathbf{e}_1, \mathbf{e}_2, \dots, \mathbf{e}_p\}$ is the collection of
standard basis vectors in $ \R^p $, then
$ d = 1$ and we believe that the true optimal solution
$\hat \ba \in \A_{\bt}$
corresponds to axis-aligned CART, then $P_{\A_t}(\kappa)=1$
for all values of $\kappa$.

The definition of
$ P_{\A_{\bt}}(\kappa) $
can also be understood as a hypothesis test.
Consider the regression model
$
y =
\beta_1 \Indicator(\ba^\Trans\bx \leq b)
+
\beta_2 \Indicator(\ba^\Trans \bx > b)
+
\varepsilon $
with independent Gaussian noise $ \varepsilon \sim N(0, \sigma^2)$.
Set the null hypothesis
$ H_0: \hat{\ba} \in \A_{\bt}$.
Then, using the likelihood ratio test with threshold proportional to $1-\kappa$,
$ P_{\A_{\bt}}(\kappa) $ is the likelihood of failing to reject the null
hypothesis.
It follows that the smaller $\kappa$ is, the more likely it is that we will
reject the null hypothesis that $ \hat{\ba} $ belongs to $ \A_{\bt}$.

The collection $ \A_{\bt} $ of candidate directions can be chosen in
many different ways; we discuss some examples next.
\begin{itemize}
  \item \textbf{Deterministic.}
  If $\A_{\bt}$ is nonrandom, then $ P_{\A_{\bt}}(\kappa) $ is either zero or
  one for any $\A_{\bt}\subset \R^p$,
  and if $ \A_{\bt}= \mathbb{R}^p $, then $ P_{\A_{\bt}}(\kappa) = 1 $ for all
  $\kappa \in (0, 1] $.
  For the latter case, one can use strategies based on
  mixed-integer optimization (MIO)
  \citet{zhu2020scalable, dunn2018optimal, bertsimas2017optimal}.
  In particular, \citet{dunn2018optimal} presents a global MIO formulation for
  regression trees with squared error that can also be implemented greedily
  within each node. Separately, in order to improve interpretability,
  it may be of interest to restrict the coordinates of $ \hat\ba $ to be
  integers.
  Using the hyperplane separation theorem and the
  fact that constant multiples of vectors in $ \mathbb{Z}^p $ are dense in
  $\R^p$, it can easily be shown that if $ \A_{\bt}= \mathbb{Z}^p $,
  then $ P_{\A_{\bt}}(\kappa) = 1 $ for all
  $\kappa \in (0, 1] $.
  An integer-valued search space may also lend itself to optimization strategies
  based on integer programming.

\item \textbf{Purely random.}
  The most na\"ive and agnostic way to construct $ \A_{\bt} $ is to generate the
  directions uniformly at random.
  For example, with axis-aligned CART where the global search space consists
  of the $ p $ standard basis vectors
  $ \{\mathbf{e}_1, \mathbf{e}_2, \dots, \mathbf{e}_p\} $,
  if $ \A_{\bt}$ is generated by selecting
  $m\; (\leq p)$ standard basis vectors
  uniformly at random without
  replacement (as is done with random forests \citep{breiman2001random}), then $
  P_{\A_{\bt}}(\kappa) \geq \binom{p-1}{m-1}/\binom{p}{m} = m/p $ for all $ \kappa \in (0, 1] $.
  For more complex global search spaces (e.g., oblique),
  it is quite likely that a purely random selection will yield very small
  $ P_{\A_{\bt}}(\kappa) $. For example, if the global search space is $\{ \ba
  \in \mathbb{R}^p : \|\ba\|_{\ell_0}=d\} $ and $ \A_{\bt} $ is generated by
  selecting $m$ (distinct) sets $S_k \subset \{1, 2, \dots, p\} $ with $|S_k|=d$
  uniformly at random without replacement and  setting $ \A_{\bt} = \bigcup_k
  \{\ba \in \mathbb{R}^p : a_j = 0,\; j\notin S_k\} $, then
  $P_{\A_{\bt}}(\kappa) \geq m/\binom{p}{d} \approx 0 $ for all $ \kappa \in (0,
  1]$.
  This has direct consequences for the predictive performance, since,
  as we shall see (Section~\ref{sec:trainerr}), the expected risk is inflated by
  the reciprocal probability $ 1/P_{\A_{\bt}}(\kappa) $.
  Thus, generating $ \A_{\bt}$ in a principled manner is important for
  producing small risk.

\item \textbf{Data-dependent.}
  Perhaps the most interesting and useful way of generating informative
  candidate directions in $\A_{\bt}$ is to take a data-driven approach.
  One possibility is to use dimensionality reduction techniques,
  such as PCA, LDA, and Lasso, on a separate sample
  $ \{(\tilde y_i, \tilde \bx^\Trans_i): \tilde\bx_i \in \bt\} $.
  The search space $\A_{\bt}$ can then be defined in terms of the top principle
  components produced by PCA or LDA,
  or, similarly, in terms of the relevant coordinates selected by Lasso.
  Additional randomization can also be introduced by incorporating,
  for example, sparse random projections or random rotations \citep{tomita2020sparse}.
  On an intuitive level, we expect these statistical methods that aim to capture
  variance in the data to produce good optimizers of the objective
  function.
  Indeed, empirical studies with similar constructions provide evidence
  for their efficacy over purely random strategies
  \citep{ghosh2021efficient, menze2011oblique, rodriguez2006rotation}.
\end{itemize}

In order to control the predictive performance of the decision tree
theoretically,
we assume the researcher has chosen a meaningful method for
selecting candidate directions
$\A_{\bt}$, either with prior knowledge based on the context of the problem,
or with an effective data-driven strategy.

\subsection{Orthogonal Tree Expansions}

We now present a technical result about the construction of trees that is
crucial in proving our main results.
While Lemma~\ref{lmm:hilbert} below focuses on the special case of constant fit
at the terminal nodes for concreteness, all proofs in Appendix~\ref{sec:proofs}
are given in full generality. To be more precise, our results in the appendix
allow for any finite-dimensional least squares fit at the terminal nodes, and
thus give a general orthogonal tree expansion in the function space for adaptive
oblique decision trees, covering canonical adaptive axis-aligned decision trees
as a special case.

Lemma~\ref{lmm:hilbert} shows that the tree output
$ \muhat(T)(\bx) $
is equal to the empirical orthogonal projection of $\by$ onto the
linear span of orthonormal decision stumps, defined as
\begin{equation}
~\label{eq:stumps_norm}
  \psi_{\bt}(\bx)
  =
  \frac{ \Indicator(\bx \in \bt_L) n(\bt_R)
  - \Indicator(\bx \in \bt_R) n(\bt_L) }
  {\sqrt{w(\bt) n(\bt_L) n(\bt_R)} },
\end{equation}
for internal nodes $ \bt \in [T] $,
where
$w(\bt) = n(\bt)/n$ denotes the proportion of observations that are in
$\bt$.
By slightly expanding the notion of an internal node
to include the empty node (i.e., the empty set), we define
$ \psi_{\bt}(\bx) \equiv 1 $ if $ \bt $ is the empty node, in which case the
tree outputs the grand mean of all the response values.
The decision stump $ \psi_{\bt} $ in~\eqref{eq:stumps_norm}
is produced from the Gram–Schmidt orthonormalization of the functions
$ \{\Indicator(\bx \in \bt), \Indicator(\bx \in \bt_L) \} $ with respect to the
empirical inner product space:
\begin{align*}
  \Bigg\{
  \frac{\Indicator(\bx \in \bt)}{\|\Indicator(\bx \in \bt)\|_n}, \;
  \tfrac{\Indicator(\bx \in \bt_L)
  - \frac{\langle \Indicator(\bx \in \bt_L), \Indicator(\bx \in \bt) \rangle_n}
  {\|\Indicator(\bx \in \bt)\|^2_n}
  \Indicator(\bx \in \bt)}
  {\big\|\Indicator(\bx \in \bt_L) -
  \frac{\langle \Indicator(\bx \in \bt_L), \Indicator(\bx \in \bt)
  \rangle_n}
  {\|\Indicator(\bx \in \bt)\|^2_n}
  \Indicator(\bx \in \bt)\big\|_n}
  \Bigg\}
  &
    =
    \Bigg\{\frac{\Indicator(\bx \in \bt)}{\sqrt{w(\bt)}}, \;
    \frac{ \Indicator(\bx \in \bt_L) n(\bt_R)
    - \Indicator(\bx \in \bt_R) n(\bt_L) }
    {\sqrt{w(\bt) n(\bt_L) n(\bt_R)} } \Bigg\}.
\end{align*}
We refer the reader to Appendix \ref{sec:proofs} for
an orthonormal decomposition of the tree output that holds
in a much more general setting
(i.e., when the node output is the least squares
projection onto the linear span of a finite dictionary).

\begin{lemma}
~\label{lmm:hilbert}
  If $ T $ is a decision tree constructed with CART methodology
  (either axis-aligned or oblique),
  then its output~\eqref{eq:output} admits the orthogonal expansion
  \begin{equation}
    \label{eq:orth}
    \muhat(T)(\bx)
    =
    \sum_{\bt \in [T]}
    \langle \by , \bpsi_{\bt} \rangle_n \psi_{\bt}(\bx),
  \end{equation}
  where $\bpsi_{\bt}= (\psi_{\bt}(\bx_1), \dots, \psi_{\bt}(\bx_n))^{\Trans} $.
  By construction,
  $ \|\bpsi_{\bt}\|_n= 1 $
  and $ \langle \bpsi_{\bt}, \bpsi_{\bt^{\prime}} \rangle_n = 0 $
  for distinct internal nodes $ \bt $ and $ \bt^{\prime} $ in $ [T] $.
  In other words, $ \bmuhat(T) $
  is the empirical orthogonal projection of $\by$ onto the linear span of
  $ \{\bpsi_{\bt}\}_{\bt \in [T]} $.
  Furthermore,
  \begin{equation}
    \label{eq:opt}
    |\langle \by, \bpsi_{\bt} \rangle_n|^2 =
    \widehat{\Delta}(\hat b, \hat \ba, \bt).
  \end{equation}
\end{lemma}

\begin{remark}[Connection to Sieve Estimation Literature]
Another way of thinking about CART is through the lens of least squares sieve estimation.
For example, for a fixed but otherwise arbitrary ordering of the internal nodes
of $T$,
suppose $\mathbf{\Psi}$ is the $ n \times |[T]|$ data matrix
$ [\psi_{\bt}(\bx_i)]_{1 \leq i \leq n, \; \bt\in[T]} $
and $\mathbf{\Psi}(\bx)$ is the
$ |[T]| \times 1 $ feature vector $ (\psi_{\bt}(\bx))_{\bt\in[T]}$.
Then,
\begin{equation*}
\muhat(T)(\bx)
  = \mathbf{\Psi}(\bx)^{\Trans}
  (\mathbf{\Psi}^{\Trans}\mathbf{\Psi})^{-1}
  \mathbf{\Psi}^{\Trans}\by
  = \mathbf{\Psi}(\bx) ^{\Trans}\mathbf{\Psi}^{\Trans}\by.
\end{equation*}
From this perspective, standard sieve estimation and
inference theory \citep{huang2003local,cattaneo2020large} cannot be applied to
studying the statistical properties of $\muhat(T)(\bx)$
because the implied (random) basis functions depend on the entire sample
$(\by,\bX)$ through the adaptive (recursive) split regions underlying the
decision tree construction (i.e., the induced random partitioning).
\end{remark}

Lemma~\ref{lmm:hilbert} suggests that there may be some connections between
oblique CART and sequential greedy optimization in Hilbert spaces.
Indeed, our analysis of the
oblique CART algorithm suggests that it can be viewed as a local orthogonal
greedy procedure in which one iteratively
projects the data onto the space of all constant predictors within a greedily
obtained node.
The algorithm also has similarities to forward-stepwise
regression because,
at each current node $ \bt $, it grows the tree by selecting a feature,
$ \bpsi_{\bt} $, most correlated with the residuals,
$(y_i - \overline y_{\bt})\Indicator(\bx_i \in \bt) $,
per~\eqref{eq:opt} and~\eqref{eq:sse},
and then adding that chosen feature along with its coefficient back to the tree
output in~\eqref{eq:orth}.

The proofs show that this local greedy approach has a very
similar structure to standard global greedy algorithms in Hilbert spaces.
Indeed, the reader familiar with greedy algorithms in Hilbert spaces for
over-complete dictionaries will recognize some similarities in the analysis
(see the \emph{orthogonal greedy algorithm} \citep{barron2008approximation}
in which one iteratively projects the data onto the linear span of a
finite collection of greedily obtained dictionary elements).
As with all orthogonal expansions, the decomposition of
$ \bmuhat(T_K) $ in Lemma~\ref{lmm:hilbert}
allows one to write down a recursive expression for the training error.
That is, from
$
\bmuhat(T_K)
= \bmuhat(T_{K-1})
+ \sum_{\bt \in T_{K-1}} \langle \by , \bpsi_{\bt} \rangle_n \bpsi_{\bt}$,
one obtains the
identity
\begin{equation}
  \label{eq:training_recursion}
  \| \by-\bmuhat(T_K)\|^2_n
  =
  \| \by-\bmuhat(T_{K-1})\|^2_n
  -
  \sum_{\bt \in T_{K-1}}|\langle \by, \bpsi_{\bt} \rangle_n|^2.
\end{equation}
Furthermore, using the fact that
$ \langle \by, \bpsi_{\bt} \rangle_n $ is the
result of a local maximization, viz.,
the equivalence~\eqref{eq:opt} in Lemma~\ref{lmm:hilbert},
one can construct an empirical probability measure
$ \Pi $ on $(b, \ba) $ and lower bound
$ |\langle \by, \bpsi_{\bt} \rangle_n|^2 $
by $ \int \widehat{\Delta}(b, \ba, \bt) \diff\Pi(b, \ba) $,
which is itself further lower bounded by an
appropriately scaled squared node-wise excess training error.
These inequalities
(formalized and proven in Appendix~\ref{sec:techlem})
can be combined with~\eqref{eq:training_recursion} to
provide a useful training error bound.
We formally present this result next.

\subsection{Training Error Bound for Oblique CART}
\label{sec:trainerr}
Applying the techniques outlined earlier, we
can show the following result (Lemma~\ref{lmm:training})
on the training error of the tree.
Our result provides an algorithmic guarantee, namely,
that the expected excess training error of a
depth $ K $ tree constructed with oblique CART methodology decays like $ 1/K $,
and, with additional assumptions (see Section~\ref{sec:fast}),
like $ 4^{-K/q} $ for some $ q > 2$.
To the best of our knowledge, this result is the first of its kind for oblique
CART.\ The math behind it is surprisingly simple; in particular, unlike past work
on axis-aligned decision trees, there is no need to directly analyze the
partition that is induced by recursively splitting, which often entails showing
that certain local (i.e., node-specific) empirical quantities concentrate around
their population level versions
\citep{scornet2015consistency, wager2018estimation,
  syrgkanis2020estimation, chi2020asymptotic}.

For the following statements, the output of a depth $ K $ tree $ T_K $
constructed with oblique CART methodology using the search spaces
$ \{\A_{\bt}: \bt \in [T] \}$ is denoted $ \bmuhat( T_K) $.
Throughout the paper, we use $ \E $ to denote the expectation
with respect to the joint distribution of the (possibly random) search spaces
$ \{ \A_{\bt} : \bt \in [T_K]\} $ and the data.

\begin{lemma}[Training error bound for oblique CART]
~\label{lmm:training}
  Let $\E[y^2\log(1+|y|)]<\infty$
  and
  $ g \in \F $
  with $ \|g\|_{\L_1} < \infty $.
  Then, for any $ K \geq 1$,
  \begin{equation}
    \label{eq:trainlem}
    \E\big[\|  \by - \bmuhat( T_K) \|^2_n\big]
    \leq
    \E\big[\|  \by - \bg \|^2_n\big]
    + \frac{\|g\|^2_{\L_1}\E\big[\max_{\bt\in [T_K]}
      P^{-1}_{\A_{\bt}}(\kappa)\big]}
    {\kappa K}.
  \end{equation}
\end{lemma}

For this result to be non-vacuous, the only additional assumption needed is that
the largest of the reciprocal probabilities,
$ P^{-1}_{\A_{\bt}}(\kappa) $, are integrable with respect to the data and
(possibly random) search spaces. A simple sufficient condition is that the
splitting probabilities are almost surely bounded away from zero, which we
record in the following assumption for future reference.
\begin{assumption}[Non-zero splitting probabilities]
  \label{as:unifprob}
  The splitting probabilities are uniformly bounded away from zero.
  That is,
  $$
    \inf_{n \geq 1}\inf_{\bt\in[T_{\text{max}}]}P_{\A_{\bt}}(\kappa) > 0 \quad a.s.
  $$
\end{assumption}

Section \ref{sec:optim framework} discusses optimization algorithms/approaches
that would satisfy Assumption \ref{as:unifprob}, and, more generally, that would guarantee
$\E\big[\max_{\bt\in [T_K]}P^{-1}_{\A_{\bt}}(\kappa)\big] < \infty$.

\subsection{Pruning}
Without proper tuning of the depth $ K $, the tree  $T_K$ can very easily become
overly complicated, causing its output $\muhat(T_K)(\bx)$ to generalize poorly
to unseen data. While one could certainly select good choices of $K $ via a
holdout method, in practice,
complexity modulation is often achieved through pruning.
We first introduce some additional concepts, and then go on to describe
such a procedure.

We say that $T$ is a pruned subtree of $T^{\prime}$,
written as $ T \preceq T^{\prime}$,
if $T$ can be obtained from $T^{\prime}$ by iteratively
merging any number of its internal nodes.
A pruned subtree of $T_{\text{max}} $ is defined as
any binary subtree of $T_{\text{max}} $ having the same root node as
$T_{\text{max}} $.
Recall that the number of terminal nodes in a tree $T$ is denoted $|T|$.
As shown in \citet[Section 10.2]{breiman1984cart},
the smallest minimizing subtree for the penalty coefficient
$\lambda = \lambda_n \geq 0 $,
\begin{equation}
  \label{eq:penalized}
  T_{\text{opt}}
  \in
  \argmin_{T \preceq T_{\text{max}}}
  \Big\{ \|\by - \bmuhat( T)\|^2_n + \lambda |T| \Big\},
\end{equation}
exists and is unique
(smallest in the sense that if $T_{\text{opt}}$
optimizes the penalized risk of~\eqref{eq:penalized},
then $ T_{\text{opt}} \preceq T $ for every pruned subtree $T$ of
$T_{\text{max}}$).
For a fixed $\lambda$, the optimal subtree $T_{\text{opt}}$ can be
found efficiently by weakest link pruning, i.e., by
successively collapsing the internal node that decreases
$\|\by - \bmuhat( T)\|^2_n$ the most, until we arrive at the
single-node tree consisting of the root node.
This method enumerates a finite list of trees for which the objective function
can then be evaluated to find the optimal subtree.
Good values of $\lambda$ can be selected using cross-validation on a holdout
subset of data, for example.
See \citet{mingers1989empirical} for a description of various pruning algorithms.

We now present our main consistency and convergence rate results for both pruned and un-pruned
oblique trees.

\subsection{Oracle Inequality for Oblique CART}
Our main result establishes an adaptive prediction risk bound
(also known as an \emph{oracle inequality})
for oblique CART under model misspecification; that
is, when the true model may not belong to $\F $. Essentially, the result shows
that oblique CART performs almost as if it was finding the best approximation of
the true model with ridge expansions, while accounting for the goodness-of-fit
and descriptive complexity relative to sample size.
To bound the integrated mean squared error (IMSE),
the training error bound from Lemma~\ref{lmm:training} is
coupled with tools from empirical process theory
\citep{gyorfi2002distribution} for studying partition-based estimators.

Our results rely on the following assumption regarding
the data generating process.
\begin{assumption}[Exponential tails of the conditional response variable]
~\label{as:dgp}
  The conditional distribution of $y$ given $ \bx $
  has exponentially decaying tails.
  That is, there exist positive constants $c_1$, $c_2$, $\gamma$, and $M$,
  such that for all $ \bx \in \X $,
  \begin{equation*}
    \Prob(|y|>B+M \mid \bx) \leq c_1\exp(-c_2B^\gamma), \quad B \geq 0.
  \end{equation*}
\end{assumption}

In particular, note that $ \gamma = 1 $ for sub-Exponential data,
$ \gamma = 2 $ for sub-Gaussian data, and $ \gamma = \infty $ for bounded data.
Using the layer cake representation for expectations, i.e.,
$ |\mu(\bx)| \leq \E[|y| \mid \bx] = \int_{0}^{\infty}\Prob(|y| \geq z \mid \bx)dz $,
Assumption~\ref{as:dgp} implies that the conditional mean is uniformly bounded:
\begin{equation}
\label{eq:uniform}
  \sup_{\bx\in\X}|\mu(\bx)|
  \leq
  M + c_1\textstyle\int_0^{\infty}\exp(-c_2z^\gamma)\diff z
  = M^{\prime}
  < \infty.
\end{equation}

\begin{theorem}[Oracle inequality for oblique trees]
~\label{thm:oracle}
Let Assumption~\ref{as:dgp} hold.
  Then, for any $ K \geq 1 $,
 \begin{equation}
 \label{eq:oracle_unprune}
 \begin{aligned}
 & \E\big[\|\mu - \muhat(T_K)\|^2\big] \\
 & \qquad
    \leq 2 \inf_{f\in\F} \Bigg\{
   \|\mu - f\|^2 +
    \frac{\|f\|_{\L_1}^2\E\big[\max_{\bt\in [T_K]} P^{-1}_{\A_{\bt}}(\kappa)\big]}
    {\kappa K} +
    C\frac{2^{K}d \log(np/d)\log^{4/\gamma}(n)}{n} \Bigg\},
    \end{aligned}
  \end{equation}
  where $ C = C(c_1,c_2,\gamma, M) $ is a positive constant.
  Furthermore, if the penalty coefficient satisfies
  $ \lambda_n \gtrsim (d/n)\log(np/d)\log^{4/\gamma}(n) $,
  then
  \begin{equation}
    \begin{aligned}
      \label{eq:oracle_prune}
      & \E\big[\|\mu - \muhat( T_{\text{opt}})\|^2\big]
      \\
      & \qquad
        \leq 2 \inf_{K \geq 1, \; f\in\F} \Bigg\{
        \|\mu - f\|^2 +
        \frac{\|f\|_{\L_1}^2
        \E\big[\max_{\bt\in [T_K]} P^{-1}_{\A_{\bt}}(\kappa)\big]}{\kappa K}
        +
        C\frac{2^K d\log(np/d)\log^{4/\gamma}(n)}{n} \Bigg\}.
    \end{aligned}
  \end{equation}
\end{theorem}

Consistency of oblique trees follows from Theorem~\ref{thm:oracle}
under the additional assumption that the splitting probabilities are bounded
away from zero (Assumption~\ref{as:unifprob}) and that the depth $K$ grows
appropriately with the sample size.

\begin{corollary}[Consistency for fixed dimension]
  \label{cor:fixed_d}
  Let Assumptions~\ref{as:unifprob} and~\ref{as:dgp} hold. If $K\asymp \log n$, then
  \begin{equation*}
    \lim_{n \to \infty}\E\big[
    \|\mu - \muhat(T_{K})\|^2
    \big]
    = 0,
  \end{equation*}
  and if the penalty coefficient satisfies
  $ \lambda_n \gtrsim (d/n)\log(np/d)\log^{4/\gamma}(n) $,
  then
  \begin{equation*}
    \lim_{n \to \infty}\E\big[
    \|\mu - \muhat(T_{\text{opt}})\|^2
    \big]
    = 0.
  \end{equation*}
\end{corollary}

While Corollary~\ref{cor:fixed_d} shows that oblique trees are consistent for
fixed dimension $p$, it does not provide a rate of convergence.
Under a few additional assumptions, however,
Theorem~\ref{thm:oracle} implies that the oblique tree is consistent
with a logarithmic rate of convergence
even when the dimension grows with the sample size.

\begin{corollary}[Consistency for possibly growing dimension]
  \label{cor:oracle_consistency}
  Let Assumptions~\ref{as:unifprob} and~\ref{as:dgp} hold
  and suppose $ \{\bmu_n\} $
  is a sequence of regression functions that belong to $ \F $ with
  $ \sup_n \|\bmu_n\|_{\L_1} < \infty $.
  Assume furthermore that
  $ d = p = O(n^{1-\xi}) $ for some $ \xi \in (0, 1) $.
  If $K \asymp \log n$, then
  \begin{equation*}
    \E\big[\|
    \mu_n - \muhat(T_{K})\|^2
    \big]
    = O\big((\log n)^{-1}\big),
  \end{equation*}
  and if the penalty coefficient satisfies
  $ \lambda_n \gtrsim (d/n)\log(np/d)\log^{4/\gamma}(n) $,
  then
  \begin{equation*}
    \E\big[
    \|\mu_n - \muhat(T_{\text{opt}})\|^2
    \big]
  = O\big((\log n)^{-1}\big).
  \end{equation*}
The results also hold trivially if $d$ and $p$ are fixed.
\end{corollary}

\begin{remark}[Connection to adaptive axis-aligned decision trees]
    By considering elements of $\G$ with $\ba_k=\be_k$ (the standard basis vectors in $ \mathbb{R}^p $) and $M=p$, we recover
    the additive library
    \[
    \F^{\text{add}}
    =
    \Bigg\{ f(\bx) =
    \sum_{j=1}^p f_j(x_j) :
    f_j : \mathbb{R}\mapsto \mathbb{R} \Bigg\}.
    \]
    Additive models have played an important role in the development of theory
    for CART. For example, \cite{scornet2015consistency} show
    consistency of axis-aligned CART for fixed dimensional additive models. More
    recent work has tried to illustrate the adaptive properties of axis-aligned CART on
    sparse additive models with growing dimensionality \citep{chi2020asymptotic,
      klusowski2022large, klusowski2020sparse, syrgkanis2020estimation}, some of
    which can be recovered as a special case of our more general theory. To see
    this, note that global optimization of the splitting criterion
    \eqref{eq:sse} is feasible with axis-aligned CART ($d=1$) and hence $ \kappa
    = 1 $ and $ P_{\A_{\bt}}(\kappa) = 1$.
    Then, according to~\eqref{eq:oracle_prune}, since $ d = 1 $,
    the pruned tree estimator is consistent for regression functions in the
    class $ \F^{\text{add}} $ even in the so-called NP-dimensionality
    regime, where
    $ \log(p) = O(n^{1-\xi}) $ for some $ \xi \in (0, 1) $.
    This result was previously established in~\citet{klusowski2022large} for
    axis-aligned CART.
\end{remark}

These sort of high dimensional consistency guarantees are not possible with
non-adaptive procedures that do not
automatically adjust the amount of smoothing along a particular dimension
according to how much the covariate affects the response variable. Such
procedures perform local estimation at a query point using data that are close
in every single dimension, making them prone to the curse of dimensionality even
if the true model is sparse
(typical minimax rates
\citep{gyorfi2002distribution} necessitate that $ p $ must grow at most
logarithmically in the sample size to ensure consistency).
This is the case with
conventional multivariate (Nadaraya-Watson or local polynomial) kernel
regression in which the bandwidth is the same for all directions,
or $k$-nearest neighbors with Euclidean distance.

\section{Fast Convergence Rates}
~\label{sec:fast}
When the model is well-specified and the response values are bounded
(i.e., $ \gamma = \infty $),
as Corollary~\ref{cor:oracle_consistency} illustrates,
the oracle inequality in~\eqref{eq:oracle_unprune} yields
relatively slow rates of convergence.
Because shallow oblique trees often compete empirically with wide neural
networks \citep{bertsimas2018optimal, bertsimas2019machine, bertsimas2021voice},
a proper mathematical theory should reflect such qualities.
It is therefore natural to compare these rates with the significantly
better $r_n = \sqrt{(p/n)\log(n)}$
rates for similar function libraries, achieved by neural
networks \citep{barron1994approximation}.
In both cases, the prediction risk
converges to zero if $ p = o(n/\log(n)) $
(or equivalently, if $ r_n = o(1) $),
but the speed differs from logarithmic to polynomial.
It is unclear whether the logarithmic rate for
oblique CART is optimal in general.
We can, however, obtain comparable rates to neural networks by granting two
assumptions.
Importantly, these assumptions only need to hold on average
(with respect to the joint distribution of the data and the search sets)
and \emph{not} almost surely for all realizations of the trees.
Because most papers that study the convergence rates of neural network
estimators proceed without regard for computational complexity, to ensure a fair
comparison, we will likewise
assume here that $ d = p $, $\kappa = 1 $, and $P_{\A_{\bt}}(\kappa) = 1$
(i.e., direct optimization of~\eqref{eq:sse}).

Our first additional assumption puts a global $ \ell_q $ constraint on the local $\L_1$
total variations of the regression function $ \mu $ across all terminal nodes of $ T_{K} $.
This is a type of regularity condition on both
the tree partition of $ \X $ and the regression function $ \mu $.
It ensures a degree of compatibility between the non-additive tree model and the
additive form of the regression function. In particular, if there existed an
(oblique) tessellation of the input space such that the target function is
piecewise constant, then the following assumption would hold trivially (i.e.,
the approximation model is correctly specified). The assumption more generally
disciplines the degree of misspecification in globally approximating the unknown
target conditional expectation function when employing adaptive oblique tree
methods.

\begin{assumption}[Aggregated $\ell_q$ variation]
~\label{as:tvbound}
The regression function $\mu$ belongs to $\F$
and there exist positive numbers $V$ and $ q > 2 $ such that,
for any $ K \geq 1$,
\begin{equation} \label{eq:tvbound}
  \E\Bigg[\sum_{\bt\in T_{K}}\|\mu\|^{q}_{\L_1(\bt)}\Bigg] \leq V^q.
\end{equation}
\end{assumption}
For fixed $K$ and finite $ \|\mu\|_{\L_1} $,
there is always some choice of $ V $ and $ q $ for
which~\eqref{eq:tvbound} is satisfied since
\[
\limsup_{q \rightarrow \infty}\Bigg(\E\Bigg[\sum_{\bt\in
  T_{K}}\|\mu\|^{q}_{\L_1(\bt)}\Bigg]\Bigg)^{1/q}
\leq
\E\Bigg[\max_{\bt \in T_{K}}\|\mu\|_{\L_1(\bt)}\Bigg] \leq \|\mu\|_{\L_1},
\]
and hence, for example,
$ \E\big[\sum_{\bt\in T_{K}}\|\mu\|^{q}_{\L_1(\bt)}\big]
\leq (2\|\mu\|_{\L_1})^q $
for $ q $ large enough, but finite.
However, this alone is not enough to validate Assumption \ref{as:tvbound}
because $q$ may depend on the sample size through its dependence on
the depth $K=K_n$.
Hence, it is important that~\eqref{eq:tvbound} hold for the same $q$
\emph{uniformly} over all depths.

It turns out that Assumption \ref{as:tvbound}
can be verified to hold for $ V = \|\mu\|_{\L_1} $
and all $ q > 2 $ when $ p = 1 $.
To see this, recall that
$ I(\ba, \bt)
=[\min_{\bx \in \bt}\ba^\Trans\bx,\; \max_{\bx \in \bt}\ba^\Trans\bx] $.
Because the collection of terminal nodes
$ \{\bt:\bt\in T_{K}\} $ forms a partition of $ \X $,
when $ p = 1 $, so does
$\{ I(\ba, \bt) : \bt \in T_{K}\} $
for
$
I(\ba, \X)
=
[\min_{\bx \in \X}\ba^\Trans\bx,\;
\max_{\bx \in \X}\ba^\Trans\bx]
$.
Thus, the $\L_1 $ total variation is additive over the nodes, i.e.,
$
\sum_{\bt\in T_{K}}\|\mu\|_{\L_1(\bt)} = \|\mu\|_{\L_1},
$
in which case,
\[
  \sum_{\bt\in T_{K}}\|\mu\|^q_{\L_1(\bt)}
  \leq
  \|\mu\|^q_{\L_1}, \quad q \geq 1.
\]
In general, for $ p > 1 $,
a crude and not very useful bound is
$ \sum_{\bt\in T_{K}}\|\mu\|^q_{\L_1(\bt)} \leq 2^K \|\mu\|^q_{\L_1}$;
however, the average size of
$ \sum_{\bt\in T_{K}}\|\mu\|^q_{\L_1(\bt)} $
will often be smaller because it depends on the specific geometry of the tree
partition of $ \X$,
which captures heterogeneity in the regression function
$\mu$. More specifically, the size will depend
on how the intervals $I(\ba, \bt)$ overlap across
$ \bt \in T_{K} $ as well as how much
$\mu$ varies within each terminal node.
We do not expect $ q $ to exceed the dimension $p$, provided that $\mu$ is smooth.
This is because, by smoothness, $ \|\mu\|_{\L_1(\bt)} $, a proxy for the  oscillation of
$\mu$ in the node is also a proxy for the diameter of the node.
Then, because the nodes are disjoint convex polytopes, on average, we expect $
\|\mu\|^p_{\L_1(\bt)} $ to be a proxy for their volume
(i.e., their $p$-dimensional Lebesgue measure),
in which case, $ \E\big[\sum_{\bt\in T_{K}}\|\mu\|^p_{\L_1(\bt)}\big] $ is a
constant multiple of the volume of $\X $.

Our final additional assumption puts a moment bound on the maximum number of
observations that any one node can contain.
Essentially, it says that the $\mathscr{L}_{\nu}$ norm of
$\max_{\bt \in T_{K}}n(\bt) $
is bounded by a multiple of the average number of observations per node.

\begin{assumption}[Node size moment bound]
~\label{as:node}
Let $ q > 2 $ be the positive number from Assumption \ref{as:tvbound}.
There exist positive numbers $A$ and $\nu \geq 1 + 2/(q-2) $ such that,
for any $ K \geq 1$,
\[
\bigg(\E\bigg[\bigg(\max_{\bt \in T_{K}}n(\bt)\bigg)^{\nu}\bigg]\bigg)^{1/\nu}
\leq \frac{An}{2^{K}}.
\]
\end{assumption}
Our risk bounds below show that $A = A_n$
is permitted to grow poly-logarithmically with the sample size,
without affecting the rate of convergence.
Because
$$
\E\bigg[\max_{\bt \in T_{K}}n(\bt)\bigg]
\leq
\bigg(\E\bigg[\bigg(\max_{\bt \in T_{K}}n(\bt)\bigg)^{\nu}\bigg]\bigg)^{1/\nu},
$$
and there are at most $2^{K}$ disjoint regions $ \bt $ in the partition of $\X$
induced by the tree at depth $ K $ such that
$ \sum_{\bt\in T_{K}}n(\bt) = n $,
Assumption~\ref{as:node} implies that, on average, no region
contains disproportionately more observations than the average number of
observations per region, i.e., $ n/2^K$.
Importantly, it still allows for situations where some regions contain very few
observations, which does tend to happen in practice.
For example, if $n=1000$, $ K = 2$, and $T_2$
has four terminal nodes with $n(\bt) \in \{5, 5, 495, 495\}  $,
then
$\max_{\bt \in T_{K}}n(\bt) \leq An/2^K $
holds with $A=2$.

Previous work by
\citet{bertsimas2018optimal} and \citet{bertsimas2019machine}
showed that feed-forward neural networks
with Heaviside activations can be transformed into
oblique decision trees with the same training error.
While these tree representations of neural networks require significant depth
(the depth of the tree in their construction is at least the width of the target
network),
they nonetheless demonstrate a proof-of-concept that supports their
extensive empirical investigations showing that the modeling power of
oblique decision trees is similar to neural networks,
even if the trees have modest depth ($K \leq 8$).
Our work not only complements these past studies,
it also addresses some of the scalability issues associated with global
optimization by theoretically validating greedy implementations.

\begin{lemma}
~\label{lmm:fasttraining}
Let $ d = p $, $\kappa = 1 $, and $P_{\A_{\bt}}(\kappa) = 1$,
and let Assumptions~\ref{as:tvbound} and~\ref{as:node} hold,
and assume $ \E[y^2\log(1+|y|)] < \infty $.
Then, for any $K\geq 1$,
\begin{equation}
  \label{eq:fasttrainlem}
  \E\big[\|  \by - \bmuhat( T_K) \|^2_n\big]
  \leq
  \E\big[\|  \by - \bmu \|^2_n\big]
  + \frac{AV^2}{4^{(K-1)/q}}.
\end{equation}
\end{lemma}

\begin{theorem}
~\label{thm:fastoracle}
Let $ d = p $, $\kappa = 1 $, and $P_{\A_{\bt}}(\kappa) = 1$,
and let Assumptions~\ref{as:dgp},~\ref{as:tvbound},
and~\ref{as:node} hold.
Then, for any $ K \geq 1 $,
\begin{equation}
  \label{eq:rate_unprune}
  \E\big[\|\mu - \muhat( T_K)\|^2\big]
  \leq
  \frac{2AV^2}{4^{(K-1)/q}}
  +
  C\frac{2^{K+1} p\log^{4/\gamma+1}(n)}{n},
\end{equation}
where $ C = C(c_1,c_2,\gamma, M) $ is a positive constant.
Furthermore, if
the penalty coefficient satisfies
$ \lambda_n \gtrsim (p/n)\log^{4/\gamma+1}(n) $, then
\begin{equation}
  \label{eq:rate_prune}
  \E\big[\|\mu - \muhat( T_{\text{opt}})\|^2\big]
  \leq
  2(2+q)
  \Bigg(\frac{AV^2}{q} \Bigg)^{q/(2+q)}
  \Bigg(\frac{Cp \log^{4/\gamma+1}(n)}{n}\Bigg)^{2/(2+q)}.
\end{equation}
\end{theorem}
As mentioned earlier,
we see from~\eqref{eq:rate_prune} that $A=A_n$ (as well as $V=V_n$)
is allowed to grow poly-logarithmically without affecting the convergence rate.
When
the response values are bounded (i.e., $\gamma = \infty$),
the pruned tree estimator
$\muhat( T_{\text{opt}})$
achieves the rate $r_n^{2/(2+q)} =((p/n)\log(n))^{2/(2+q)}$,
which, when $ q \approx 2 $,
is nearly identical to the $ \sqrt{r_n} $ rate in~\citet{barron1994approximation}
for neural network estimators of regression functions
$ \mu \in \F$ with $\|\mu\|_{\L_1} < \infty$.
While we make two additional assumptions
(Assumptions~\ref{as:tvbound} and~\ref{as:node})
in order for oblique CART to achieve full modeling power on par
with neural networks, our theory suggests that decision trees
might be preferred in applications where interpretability is valued,
without suffering a major loss in predictive accuracy. We also see from these
risk bounds that $ q $ plays the role of an effective dimension, since it—and
not the ambient dimension $ p $—governs the convergence rates. As we have
argued above, if $\mu$ is smooth, then $q$ should be at most $ p $, and so the
convergence rate in~\eqref{eq:rate_prune} should always be at least as fast as the
minimax optimal rate $ (1/n)^{2/(2+p)} $ for smooth functions in $p$ dimensions.

\section{Oblique Random Forests}

A random forest is a randomized ensemble of trees. While traditional random
forests use axis-aligned trees, it is also possible to work with oblique trees.

The randomization mechanism in a random forest affects the way each tree is
constructed, and consists of two parts. The first part generates a subsample without replacement of
size $N < n$ from the original training data, on which the tree is trained,
and the second part generates a random collection of candidate splitting
directions at each node, from which the optimal one is chosen (see the
discussion under the \emph{purely random} heading in Section \ref{sec:main} for
generating $ \A_{\bt} $).

Let $ \Theta $ denote the random variable whose law governs the aforementioned
randomization mechanism and let $T_K(\Theta)$ be the associated maximal tree of
depth $K$. Let $ \boldsymbol{\Theta} = (\Theta_1, \dots, \Theta_B)^{\Trans} $
denote $ B $ independent copies of $\Theta $, corresponding to $ B $ trees $
T_K(\Theta_b) $, for $ b = 1, \dots, B $. The output of the random forest at a
point $ \bx $ is obtained by averaging the predictions of all $ B $ trees in the
forest, viz.,
\begin{equation*}
  \muhat(\boldsymbol{\Theta})(\bx)
  =
  \frac{1}{B}\sum_{b=1}^B
  \widehat\mu(T_K(\Theta_b))(\bx).
\end{equation*}
By convexity of squared error loss, the expected risk can be bounded as follows:
\begin{equation*}
  \E\big[\|\mu-\widehat\mu(\boldsymbol{\Theta})\|^2\big]
  \leq
  \frac{1}{B}\sum_{b=1}^B\E\big[\|\mu-\widehat\mu(T_K(\Theta_b))\|^2\big]
  = \E\big[\|\mu-\widehat\mu(T_K(\Theta))\|^2\big].
\end{equation*}
The above bound, although crude, tells us that we should expect the random
forest to perform no worse than a single (random) tree.

\subsection{Oracle inequality for oblique forests}
~\label{sec:orfmse}
We can now establish an oracle inequality for oblique forests similar to that of
Theorem~\ref{thm:oracle}.
Conditional on the randomness due to the indices $ \I \subset \{1, \ldots, n\}$
of the original training data that belong to
the subsampled training data, $ \widehat\mu(T_K(\Theta_b)) $ is a depth $ K $ oblique
tree (with randomized splits) trained on $ N $ samples for each draw $b=1,\dots,B$.
This means that
$\E\big[\|\mu-\widehat\mu(T_K(\Theta_b))\|^2 \mid \I \big]$
enjoys the \emph{exact} same bounds in Theorem \ref{thm:oracle}
but with $n $ replaced by the effective sample size $ |\I | = N $.
We formalize this notion in Theorem~\ref{thm:orf_oracle}.

\begin{theorem}[Oracle inequality for oblique forests]
~\label{thm:orf_oracle}
  Suppose Assumptions~\ref{as:dgp} holds.
  Let $\muhat(\boldsymbol{\Theta})$ be the output of the oblique random forest constructed
  with oblique trees of depth $ K $.
  Then,
  \begin{equation*}
    \E \big[\|\mu -\muhat(\boldsymbol{\Theta})\|^2\big]
    \leq 2 \inf_{f\in\F} \Bigg\{ \|\mu - f\|^2 +
    \frac{\|f\|_{\L_1}^2\E\big[\max_{\bt\in [T_K]} P^{-1}_{\A_{\bt}}(\kappa)\big]}
    {\kappa K} +
    C\frac{2^{K}d \log(Np/d)\log^{4/\gamma}(N)}{N} \Bigg\},
  \end{equation*}
  where $C$ is some positive constant
  and $N$ is the subsample size.
\end{theorem}

While the efficacy of forests is not reflected in these risk bounds,
they do show that forests of oblique trees inherit the same desirable properties as single trees.
It should be noted that the expectation in the second term of the bound in
Theorem~\ref{thm:orf_oracle} is over the subsampled data (instead of over the
entire data set as in Theorem~\ref{thm:oracle}). As such, for consistency results
similar to those in Corollary~\ref{cor:fixed_d}
and~\ref{cor:oracle_consistency},
the splitting probabilities would need to be almost surely bounded away from zero
(Assumption~\ref{as:unifprob}) for any realization of the subsampled data.
Additionally, with the stronger assumptions analogous results to
Theorem~\ref{thm:fastoracle} can also be derived for oblique forests.
We omit details to conserve space.

\section{Conclusion and Future Work}

We explored how oblique decision trees---which output constant
averages over polytopal partitions of the feature space---can be used for
predictive modeling with ridge expansions,
sometimes achieving the same convergence rates as neural networks.
The theory presented here is encouraging as it implies that
interpretable models can exhibit provably good performance similar to
their black-box counterparts such as neural networks.
The computational bottleneck still remains the main obstacle for practical
implementation. Crucially, however, our risk bounds show that
favorable performance can occur even if the optimization is only done
approximately.
We conclude with a discussion of some directions for potential future research.

\subsection{Multi-layer Networks}
We can go beyond approximating single-hidden layer neural networks if instead the split
boundaries of the oblique trees have the form $ \ba^{\Trans}\boldsymbol{\Phi}(\bx) = b $, where $ \boldsymbol{\Phi} $ is a
multi-dimensional feature map, such as the output layer of a neural network. For example, if
$ \boldsymbol{\Phi}_k(\bx) = \phi(\ba_k^{\Trans}\bx-b_k)  $,
where $ \phi $ is some activation function, then this additional flexibility
allows us to approximate two-hidden layer networks, i.e., functions of the form
$ \sum_{k_2}c_{k_2}\phi(\sum_{k_1}c_{k_1,k_2}
\phi(\ba_{k_1,k_2}^{\Trans}\bx-b_{k_1,k_2})) $.

\subsection{Classification}
While we have focused on regression trees,
oblique decision trees are commonly applied to the problem of binary
classification, i.e., $ y_i \in \{-1, 1\} $. In this case, because Gini impurity \citep{Hastie-Tibshirani-Friedman2009_book, breiman1984cart} is equivalent to the squared error criterion \eqref{eq:sse}, our results also directly apply to the classification setting provided the conditional class probability $ \eta(\bx) = \mathbb{P}(y = 1 \mid \bx) $ belongs to $ \F $ and has finite $\|\eta\|_{\mathcal{L}_1}$. A more natural assumption when modeling \emph{probabilities}, however, would be to have the log-odds $ f(\bx) = \log(\eta(\bx)/(1-\eta(\bx))) $ belong to $ \F $ and have finite $\|f\|_{\mathcal{L}_1}$. In this case, we must use another widely used
splitting criterion, the \emph{information gain}, namely, the amount by which
the binary entropy of the class probabilities in the node can be reduced from
splitting the parent node \citep{Hastie-Tibshirani-Friedman2009_book, quinlan1993programs}:
\[
\text{IG}(b, \ba, \bt)
=
H(\bt)
- \frac{n(\bt_L)}{n(\bt)}H(\bt_L)
- \frac{n(\bt_R)}{n(\bt)}H(\bt_R),
\]
where
$ H(\bt) = \eta(\bt)\log(1/\eta(\bt)) + (1-\eta(\bt))\log(1/(1-\eta(\bt))) $
and
$ \eta(\bt) = \frac{1}{n(\bt)}\sum_{\bx_i \in \bt}\Indicator(y_i = 1) $.
Interestingly, maximizing
the information gain in the node is equivalent to
minimizing the node-wise logistic loss with respect to the family of log-odds
models of the form $ \theta_{\bt}(\bx) = \beta_1\Indicator(\ba^\Trans\bx \leq b) +
\beta_2 \Indicator(\ba^\Trans\bx > b)$; that is,
\begin{equation*}
  (\hat b, \hat \ba)
  \in
  \argmax_{(b, \ba)}\text{IG}(b, \ba, \bt)
  \quad \iff
  \quad
  (\hat \beta_1, \hat\beta_2, \hat b, \hat\ba)
  \in
  \argmin_{(\beta_1, \beta_2, b, \ba)}
  \sum_{\bx_i \in \bt}
  \log(1+\exp(-y_i \theta_{\bt}(\bx_i))).
\end{equation*}
One can use techniques from \citet{klusowski2022large}, which exploits connections to sequential greedy
algorithms for other convex optimization problems \citep{zhang2003sequential}
(e.g., LogitBoost), to establish a training error bound (with respect to
logistic loss) akin to Lemma~\ref{lmm:training}.
\appendix

\section{Proofs}
\label{sec:proofs}
The main text presented theory for oblique trees that output a constant
(sample average) at each node.
Fortunately, most of our results hold in a much more general setting.
In particular, we can allow for the nodes to output
$ \hat{y}_{\bt}
\in \argmin_{h\in\text{span}(\calH)}\sum_{\bx_i \in \bt}(y_i - h(\bx_i))^2 $,
where $ \calH $ is a finite dictionary that contains the
constant function.
The proofs here deal with the general case.

In what follows, we assume without loss of generality that the infimum in the
definition of $\|f\|_{\L_1}$ for $ f \in \F $ is achieved at some element
$g \in \G$,
since otherwise,
there exists $ g \in \G $ with $ \|f-g\| $
arbitrarily small and $\|g\|_{\L_1} $
arbitrarily close to
$ \|f\|_{\L_1} $.
We denote the supremum norm of a function $f: \X \mapsto \R $ by
$\|f\|_{\infty} = \sup_{\bx \in \X} |f(\bx)|$. Additionally, we slightly abuse notation by taking
$ \by-\hat{y}_{\bt} $ to mean $ \by-\hat{y}_{\bt}\Indc $,
where $ \mathbf{1} = (1, \dots, 1)^{\Trans} $ is the $n\times 1$ vector of ones.

\begin{proof}[Proof of Lemma~\ref{lmm:hilbert}]
Set
$ \mathcal{U}_{\bt}
= \big\{ u(\bx)\Indicator(\bx \in \bt_L)
+ v(\bx)\Indicator(\bx \in \bt_R) : u, v \in \text{span}(\calH)\big\} $
and consider the closed subspace
$ \mathcal{V}_{\bt}
= \big\{ v(\bx)\Indicator(\bx \in \bt): v  \in \text{span}(\calH)\big\} $.
By the orthogonal decomposition property of Hilbert spaces,
we can express $ \mathcal{U}_{\bt} $ as the direct sum
$ \mathcal{V}_{\bt}\oplus \mathcal{V}^{\perp}_{\bt} $,
where
$ \mathcal{V}^{\perp}_{\bt}
= \{ u \in \mathcal{U}_{\bt} : \langle u, v \rangle_n = 0, \;
\text{for all} \; v \in \mathcal{V}_{\bt} \} $.
Let $ \Psi_{\bt} $ be any orthonormal basis for
$ \mathcal{V}_{\bt} $ that includes
$ w^{-1/2}(\bt)\Indicator(\bx \in \bt)$,
where we remind the reader that $ w(\bt) = n(\bt)/n $.
Let $ \Psi_{\bt}^{\perp} $ be any
orthonormal basis for $ \mathcal{V}^{\perp}_{\bt} $
that includes the decision stump~\eqref{eq:stumps_norm}.
We will show that
\begin{equation}
  \label{eq:gendecomp}
  \muhat(T)(\bx)
  = \sum_{\bt \in [T]}\sum_{\psi \in \Psi_{\bt}^{\perp}}
  \langle \by, \bpsi \rangle_n \psi(\bx),
\end{equation}
where $ \{ \psi \in \Psi^{\perp}_{\bt}: \bt \in [T] \} $ is an
orthonormal dictionary, and, furthermore, that
\begin{equation}
  \label{eq:impuritydec}
  \sum_{\psi \in \Psi_{\bt}^{\perp}}
  | \langle \by, \bpsi \rangle_n|^2
  = \widehat\Delta(\hat b, \hat \ba, \bt).
\end{equation}
These identities are the respective generalizations of \eqref{eq:orth} and
\eqref{eq:opt}.
Because $ \hat{y}_{\bt}(\bx) $ is the projection of $ \by $ onto
$ \mathcal{V}_{\bt} $,
it follows that
$
\hat{y}_{\bt}(\bx)
=
\sum_{\psi \in \Psi_{\bt}}
\langle \by, \bpsi \rangle_n \psi(\bx).
$
For similar reasons,
$
\hat{y}_{\bt_L}(\bx)\Indicator(\bx \in \bt_L)
+ \hat{y}_{\bt_R}(\bx)\Indicator(\bx \in \bt_R)
= \sum_{\psi \in \Psi_{\bt}\cup \Psi_{\bt}^{\perp}}
\langle \by, \bpsi \rangle_n \psi(\bx).
$

To prove the identity in \eqref{eq:gendecomp}
(and, as a special case,~\eqref{eq:orth}),
using the above expansions,
observe that for each internal node $\bt$,
\begin{equation}
  \label{eq:stumpinnerprod}
  \sum_{\psi \in \Psi_{\bt}^{\perp}}
  \langle \by, \bpsi \rangle_n \psi(\bx)
  = (\hat{y}_{\bt_L}(\bx)-\hat{y}_{\bt}(\bx))\Indicator(\bx \in \bt_L)
  + (\hat{y}_{\bt_R}(\bx)-\hat{y}_{\bt}(\bx))\Indicator(\bx \in \bt_R).
\end{equation}
For each $ \bx \in \X $,
let $\bt_0, \bt_1, \ldots, \bt_{K-1}, \bt_K = \bt$
be the unique path from the root node $\bt_0$
to the terminal node $\bt$ that contains $ \bx $.
Next, sum~\eqref{eq:stumpinnerprod}
over all internal nodes and telescope the
successive internal node outputs to obtain
\begin{equation}
  \label{eq:telescopingsum}
  \sum_{k=0}^{K-1}
  ( \hat{y}_{\bt_{k+1}}(\bx) - \hat{y}_{\bt_k}(\bx) )
  =
  \hat{y}_{\bt_K}(\bx)
  -
  \hat{y}_{\bt_0}(\bx)
  =
  \hat{y}_{\bt}(\bx)
  -
  \hat{y}(\bx),
\end{equation}
where $ \hat{y} \in \argmin_{h \in \calH}\sum_{i=1}^n(y_i - h(\bx_i))^2 $.
Combining~\eqref{eq:stumpinnerprod} and~\eqref{eq:telescopingsum}, we have
\begin{equation*}
  \sum_{\bt \in T}
  \hat{y}_{\bt}(\bx) \Indicator(\bx \in \bt)
  =
  \hat{y}(\bx)
  +
  \sum_{\bt \in [T]\setminus \{\bt_0\}}\sum_{\psi \in \Psi_{\bt}^{\perp}}
  \langle \by, \bpsi \rangle_n \psi(\bx)
  = \sum_{\bt \in [T]}\sum_{\psi \in \Psi_{\bt}^{\perp}}
  \langle \by, \bpsi \rangle_n \psi(\bx),
\end{equation*}
where we recall that the null node $ \bt_0 $
is an internal node of $ T $.
Next, we show that
$\{\psi \in \Psi_{\bt}^{\perp} : \bt \in [T] \}$
is orthonormal.
The fact that each $ \psi $ has unit norm,
$\|\bpsi\|_n^2 = 1$, is true by definition.
If $ \psi, \psi^{\prime} \in \Psi_{\bt}^{\perp} $,
then by definition,
$ \langle \bpsi, \bpsi^{\prime}\rangle_n = 0 $.
Let $ \bt $ and $ \bt^{\prime} $ be two distinct internal nodes
and suppose
$ \psi \in \Psi_{\bt}^{\perp} $
and $ \psi^{\prime} \in \Psi_{\bt^{\prime}}^{\perp} $.
If $ \bt \cap \bt^{\prime}  = \emptyset $,
then orthogonality between $ \psi $ and $ \psi^{\prime} $ is immediate,
since $ \psi(\bx) \cdot \psi^{\prime}(\bx) \equiv 0 $.
If $ \bt \cap \bt^{\prime}  \neq \emptyset $, then due to the nested property of
the nodes,
either $ \bt \subseteq \bt^{\prime} $ or $ \bt^{\prime} \subseteq \bt $.
Assume without loss of generality that $ \bt \subseteq \bt^{\prime} $.
Then $ \psi^{\prime}  $,
when restricted to $ \bx \in \bt $,
belongs to $ \mathcal{V}_{\bt} $,
which also implies that $ \psi $ and $ \psi^{\prime} $ are orthogonal,
since $ \psi \in \mathcal{V}_{\bt}^{\perp} $.

Finally, the decrease in impurity identity~\eqref{eq:impuritydec}
(and, as a special case,~\eqref{eq:opt})
can be shown as follows:
\begin{align*}
  \widehat\Delta(\hat b, \hat \ba, \bt)
  & =
    \frac{1}{n}\sum_{\bx_i \in \bt}
    (y_i - \hat{y}_{\bt}(\bx_i))^2
    - \frac{1}{n}\sum_{\bx_i \in \bt}
    (y_i - \hat{y}_{\bt_L}(\bx_i)\Indicator(\bx_i \in \bt_L)
    - \hat{y}_{\bt_R}(\bx_i)\Indicator(\bx_i \in \bt_R))^2
  \\
  & = \Bigg(\frac{1}{n}\sum_{\bx_i \in \bt}y^2_i
    -\sum_{\psi \in \Psi_{\bt}}
    |\langle \by, \bpsi \rangle_n|^2\Bigg)
    - \Bigg(\frac{1}{n}\sum_{\bx_i \in \bt}y^2_i
    -\sum_{\psi \in \Psi_{\bt}\cup \Psi_{\bt}^{\perp}}
    |\langle \by, \bpsi \rangle_n|^2\Bigg)
  \\
  & = \sum_{\psi \in \Psi_{\bt}^{\perp}}
    |\langle \by, \bpsi \rangle_n|^2.\qedhere
\end{align*}
\end{proof}

Throughout the remaining proofs, we will assume that there exists a positive constant
$ Q \geq 1 $ such that
$ \sup_{\bx\in \X}|\muhat(T)(\bx)| \leq Q \cdot \sqrt{\max_{1 \leq i\leq
    n}\frac{1}{i}\sum_{\ell=1}^i y^2_{\ell}} $,
almost surely. This assumption is drawn from the bound
\begin{equation*}
  |\hat{y}_{\bt}(\bx)|
  \leq
  \sqrt{\max_{1 \leq i\leq n}\frac{1}{i}\sum_{1\leq \ell\leq i} y^2_{\ell}}
  \sqrt{\vphantom{\frac{1}{i}}{w(\bt)\sum_{\psi \in \Psi_{\bt}}\psi^2(\bx)}},
\end{equation*}
which is established by first using the basis expansion for $ \hat{y}_{\bt} $
provided in the proof of Lemma~\ref{lmm:hilbert} and the Cauchy-Schwarz
inequality,
\begin{equation}
  |\hat{y}_{\bt}(\bx)|
  =
  \Bigg|\sum_{\psi \in \Psi_{\bt}}
  \langle \by, \psi \rangle_n \psi(\bx)\Bigg|
  \\
  \leq
  \sqrt{\sum_{\psi \in \Psi_{\bt}}|\langle \by, \bpsi \rangle_n|^2}
  \sqrt{\sum_{\psi \in \Psi_{\bt}}\psi^2(\bx)},
\end{equation}
and then, because $ \{\psi : \psi \in \Psi_{\bt} \} $ is orthonormal,
employing Bessel's inequality to obtain
$ \sum_{\psi \in
\Psi_{\bt}}|\langle \by, \bpsi \rangle_n|^2 \leq
n^{-1}\sum_{\bx_i\in\bt}y_i^2 \leq w(\bt)\max_{1 \leq i\leq
  n}\frac{1}{i}\sum_{\ell=1}^i y^2_{\ell} $.
Thus,
$Q$ could be taken to equal (or be an almost sure bound on) $
\sup_{\bx\in\X}\max_{\bt \in [T]} \sqrt{w(\bt)\sum_{\psi \in
\Psi_{\bt}}\psi^2(\bx)} $. In the conventional case where the
tree outputs a constant in each node, $ \Psi_{\bt} = \big\{
w^{-1/2}(\bt)\Indicator(\bx \in \bt) \big\} $, and hence $ Q = 1$.
To ensure that $ \muhat(T)(\bx) $ is
square-integrable,
i.e., $ \E\big[\sup_{\bx\in\X}|\muhat(T)(\bx)|^2\big] < \infty $,
we merely need to check that
$ \E\big[\max_{1 \leq i \leq n}\frac{1}{i}\sum_{\ell=1}^i
y_{\ell}^2\big]<\infty$. This
follows easily from Doob's maximal inequality for positive sub-martingales
\citep[Theorem 5.4.4]{durrett_2019},
since $ \E[y^2\log(1+|y|)] < \infty $ by assumption.

\begin{proof}[Proof of Lemmas~\ref{lmm:training} and~\ref{lmm:fasttraining}]
Define the excess training error as
\[
R_{K} = \|\by-\bmuhat( T_{K})\|^{2}_n - \|\by-\bg\|^{2}_{n}.
\]

Define the squared node-wise norm and node-wise inner product as
$ \|\bff\|^2_{\bt}
= \frac{1}{n(\bt)}\sum_{\bx_i \in \bt} (f(\bx_i))^2 $
and
$ \langle \bff, \bg\rangle_{\bt} = \frac{1}{n(\bt)}\sum_{\bx_i \in \bt}
f(\bx_i)g(\bx_i)$,
respectively.
We define the node-wide excess training error as
\[
R_{K}(\bt) = \|\by-\hat{y}_{\bt}\|^{2}_{\bt} - \|\by-\bg\|^{2}_{\bt}.
\]
We use this to rewrite the total excess training error as a weighted combination
of the node-wide excess train errors:
\[
  R_{K} = \sum_{\bt \in T_{K}} w(\bt)R_{K}(\bt),
  \quad w(\bt) = n(\bt)/n,
\]
where $\bt \in T_{K}$ means $\bt$ is a terminal node of $T_{K}$.
From the orthogonal decomposition of the tree,
as given in~\eqref{eq:gendecomp},
we have

\begin{equation}
  \label{eq:decomposition}
  \|\by-\bmuhat( T_{K})\|_{n}^{2}
  =
  \|\by-\bmuhat( T_{K-1})\|_{n}^{2}
  -
  \sum_{\bt \in T_{K-1}}\sum_{\psi \in \Psi_{\bt}^{\perp}}
  |\langle \by, \bpsi \rangle_n|^2.
\end{equation}

Subtracting $\|\by-\bg\|^{2}_n$ on both sides of~\eqref{eq:decomposition},
and using the definition of $R_{K}$, we obtain
\begin{equation}
  \label{eq:rkid}
  R_{K}
  =
  R_{K-1}
  -
  \sum_{\bt \in T_{K-1}}\sum_{\psi \in \Psi_{\bt}^{\perp}}
  |\langle \by, \bpsi \rangle_n|^2.
\end{equation}
Henceforth, we adopt the notation $\E_{T_K}[R_K]$ to mean that the expectation
is taken with respect to the joint distribution of
$ \{ \A_{\bt} : \bt \in [T_K]\} $,
conditional on the data.
We can assume
$\E[R_K]>0$ for all $ K \geq 1 $, since otherwise, by definition of $R_K$,
\begin{equation*}
  \E[R_K]
  =
  \E[ \|\by-\bmuhat( T_{K})\|^{2}_n - \|\by-\bg\|^{2}_{n} ]
  \leq 0,
\end{equation*}
which directly gives the desired result.

Using the law of iterated expectations
and the recursive relationship obtained in~\eqref{eq:rkid},
\begin{equation}
  \label{eq:trainingdecomp}
  \E_{T_K}[R_{K}]
  =
  \E_{T_{K-1}}[ \E_{T_K|T_{K-1}} [ R_{K} ] ]
  =
  \E_{T_{K-1}}[ R_{K-1} ]
  -
  \E_{T_{K-1}}
  \Bigg[
  \E_{T_K|T_{K-1}}
  \Bigg[
  \sum_{\bt \in T_{K-1}}
  \sum_{\psi \in \Psi_{\bt}^{\perp}}
  |\langle \by, \bpsi \rangle_n|^2
  \Bigg]
  \Bigg].
\end{equation}
By \eqref{eq:impuritydec}
and the sub-optimality probability, $P_{\A(\bt)}(\kappa)$,
we can rewrite the term inside the iterated expectation in~\eqref{eq:trainingdecomp} as
\begin{equation} \label{eq:impuritylower}
\begin{aligned}
  \sum_{\bt \in T_{K-1}}
  \sum_{\psi \in \Psi_{\bt}^{\perp}}
  |\langle \by, \bpsi \rangle_n|^2
  &
    =
    \sum_{\bt\in T_{K-1}}
    \widehat{\Delta}(\hat{b}, \hat{\ba}, \bt) \\
    & \geq
    \sum_{\bt\in T_{K-1}}
    \Indicator
    \Big(
    \widehat{\Delta}(\hat{b}, \hat{\ba}, \bt)
    \geq
    \kappa \max_{(b, \ba) \in \mathbb{R}^{1+p}}\widehat{\Delta}(b, \ba, \bt)
    \Big)
    \widehat{\Delta}(\hat b, \hat \ba, \bt) \\
    & \geq
    \kappa
    \sum_{\bt\in T_{K-1}}
    \Indicator
    \Big(
    \widehat{\Delta}(\hat{b}, \hat{\ba}, \bt)
    \geq
    \kappa \max_{(b, \ba) \in \mathbb{R}^{1+p}}\widehat{\Delta}(b, \ba, \bt)
    \Big)
    \max_{(b, \ba) \in \mathbb{R}^{1+p}}\widehat{\Delta}(b, \ba, \bt).
\end{aligned}
\end{equation}
Taking expectations of both sides of \eqref{eq:impuritylower}
with respect to the conditional distribution of $ T_{K} $ given $ T_{K-1} $,
we have
\begin{equation}
\label{eq:deltalowerind}
\begin{aligned}
  &
    \E_{T_K|T_{K-1}}\Bigg[
    \sum_{\bt \in T_{K-1}}
    \sum_{\psi \in \Psi_{\bt}^{\perp}}
    |\langle \by, \bpsi \rangle_n|^2
    \Bigg]
  \\
  &
   \qquad
    \geq
    \kappa
    \sum_{\bt\in T_{K-1}}
    \E_{T_K|T_{K-1}}\Big[
    \Indicator
    \Big(
    \widehat{\Delta}(\hat{b}, \hat{\ba}, \bt)
    \geq
    \kappa \max_{(b, \ba) \in \mathbb{R}^{1+p}}\widehat{\Delta}(b, \ba, \bt)
    \Big)
    \max_{(b, \ba) \in \mathbb{R}^{1+p}}\widehat{\Delta}(b, \ba, \bt)
    \Big].
\end{aligned}
\end{equation}
By definition of $P_{\A(\bt)}$,
\begin{equation} \label{eq:deltalowerprob}
  \begin{aligned}
    & \sum_{\bt\in T_{K-1}}
      \E_{T_K|T_{K-1}}\Big[
      \Indicator
      \Big(
      \widehat{\Delta}(\hat{b}, \hat{\ba}, \bt)
      \geq
      \kappa \max_{(b, \ba) \in \mathbb{R}^{1+p}}\widehat{\Delta}(b, \ba, \bt)
      \Big)
      \max_{(b, \ba) \in \mathbb{R}^{1+p}}\widehat{\Delta}(b, \ba, \bt)
      \Big]
    \\
    &
      =
      \sum_{\bt\in T_{K-1}}
      P_{\A_{\bt}}(\kappa)
      \max_{(b, \ba) \in \mathbb{R}^{1+p}}\widehat{\Delta}(b, \ba, \bt)
    \\
    &
      \geq
      \sum_{\bt\in T_{K-1}: R_{K-1}(\bt) > 0}
      P_{\A_{\bt}}(\kappa)
      \max_{(b, \ba) \in \mathbb{R}^{1+p}}\widehat{\Delta}(b, \ba, \bt).
  \end{aligned}
\end{equation}
In turn, by
Lemma~\ref{lem:impurity_bound},
\begin{equation}
  \label{eq:deltalowerrisk}
  \sum_{\bt\in T_{K-1}: R_{K-1}(\bt) > 0}
  P_{\A_{\bt}}(\kappa)
  \max_{(b, \ba) \in \mathbb{R}^{1+p}}\widehat{\Delta}(b, \ba, \bt)
  \geq
  \sum_{\bt \in T_{K-1}: R_{K-1}(\bt) >0}
  w(\bt)
  \frac{ R^2_{K-1}(\bt) }{P^{-1}_{\A_{\bt}}(\kappa)\|g\|^{2}_{\L_1(\bt)} },
\end{equation}
and Lemma~\ref{lmm:sedrakyan},
\begin{equation}
\label{eq:expbound}
\begin{aligned}
  \sum_{\bt \in T_{K-1}: R_{K-1}(\bt) >0}
  w(\bt)
  \frac{ R^2_{K-1}(\bt) }{P^{-1}_{\A_{\bt}}(\kappa)\|g\|^{2}_{\L_1(\bt)} }
  &
    \geq
    \frac{( \sum_{\bt\in T_{K-1}: R_{K-1}(\bt) >0}w(\bt) R_{K-1}(\bt) )^2}
    {\sum_{\bt \in T_{K-1}: R_{K-1}(\bt) >0}
    w(\bt)P^{-1}_{\A_{\bt}}(\kappa) \|g\|^2_{\L_1(\bt)}}
  \\
  &
    \geq
    \frac{(R^{+}_{K-1} )^2}
    {\sum_{\bt \in T_{K-1}} w(\bt)P^{-1}_{\A_{\bt}}(\kappa) \|g\|^2_{\L_1(\bt)}},
\end{aligned}
\end{equation}
where
$ R^{+}_{K-1}
= \sum_{\bt\in T_{K-1}: R_{K-1}(\bt) >0}w(\bt) R_{K-1}(\bt) \geq R_{K-1}$.
Combining~\eqref{eq:deltalowerind},~\eqref{eq:deltalowerprob},~\eqref{eq:deltalowerrisk},
and~\eqref{eq:expbound}
and plugging the result into~\eqref{eq:trainingdecomp},
we obtain
\begin{equation*}
  \E_{T_K}[R_{K}]
    \leq
    \E_{T_{K-1}}[R_{K-1}]
    -
    \kappa
    \E_{T_{K-1}}
    \Bigg[
    \frac{(R^{+}_{K-1})^2}
    {\sum_{\bt \in T_{K-1}} w(\bt)P^{-1}_{\A_{\bt}}(\kappa) \|g\|^2_{\L_1(\bt)}}
    \Bigg].
\end{equation*}
Using Lemma~\ref{lmm:sedrakyan} again, we have
\begin{equation*}
  \E_{T_{K-1}}
  \Bigg[
  \frac{(R^{+}_{K-1})^2}
  {\sum_{\bt \in T_{K-1}} w(\bt)P^{-1}_{\A_{\bt}}(\kappa) \|g\|^2_{\L_1(\bt)}}
  \Bigg]
  \geq
  \frac{(\E_{T_{K-1}}[R^{+}_{K-1}])^2}
  {\E_{T_{K-1}}\Big[
  \sum_{\bt \in T_{K-1}} w(\bt)P^{-1}_{\A_{\bt}}(\kappa) \|g\|^2_{\L_1(\bt)}
  \Big]}.
\end{equation*}
We have therefore derived the recursion
\begin{equation}
  \label{eq:recursion}
  \E_{T_K}[R_{K}]
  \leq
  \E_{T_{K-1}}[R_{K-1}]
  -
  \kappa
  \frac{(\E_{T_{K-1}}[R^{+}_{K-1}])^2}
  {\E_{T_{K-1}}\Big[\sum_{\bt \in T_{K-1}}
  w(\bt)P^{-1}_{\A_{\bt}}(\kappa) \|g\|^2_{\L_1(\bt)}\Big]}.
\end{equation}
Next, let us take the expectation of both sides of~\eqref{eq:recursion}
with respect to the data,
apply Lemma~\ref{lmm:sedrakyan} once again,
and use the fact that
$ R^{+}_{K-1} \geq R_{K-1} $
and $ \E[R_{K-1}] > 0 $ to obtain
\begin{align*}
  \E[R_{K}]
  &
  \leq
  \E[R_{K-1}]
  -
  \kappa
  \E\Bigg[\frac{(\E_{T_{K-1}}[R^{+}_{K-1}])^2}
  {\E_{T_{K-1}}\Big[
  \sum_{\bt \in T_{K-1}} w(\bt)P^{-1}_{\A_{\bt}}(\kappa) \|g\|^2_{\L_1(\bt)}
  \Big]}\Bigg]
  \\
  &
  \leq \E[R_{K-1}]
  -
  \kappa
  \frac{(\E[R^{+}_{K-1}])^2}
  {\E\Big[
  \sum_{\bt \in T_{K-1}} w(\bt)P^{-1}_{\A_{\bt}}(\kappa) \|g\|^2_{\L_1(\bt)}
  \Big]}
  \\
  &
  \leq
  \E[R_{K-1}]
  -
  \kappa
  \frac{(\E[R_{K-1}])^2}
  {\E\Big[
  \sum_{\bt \in T_{K-1}} w(\bt)P^{-1}_{\A_{\bt}}(\kappa) \|g\|^2_{\L_1(\bt)}
  \Big]}.
\end{align*}
We have therefore obtained a recursion for $ \E[R_{K}] $,
which we can now solve thanks to Lemma~\ref{lem:rec}.
Setting
$a_k = \E[R_k]$
and
$ b_k
= \kappa/
\E\big[\sum_{\bt \in T_{k-1}} w(\bt)P^{-1}_{\A_{\bt}}(\kappa)
\|g\|^2_{\L_1(\bt)}\big] $
in Lemma~\ref{lem:rec}, we have
\begin{equation} \label{eq:solution}
\E[R_K]
\leq
\frac{1}{\kappa\sum_{k=1}^K 1/
\E\big[\sum_{\bt \in T_{k-1}} w(\bt)P^{-1}_{\A_{\bt}}(\kappa)
\|g\|^2_{\L_1(\bt)}\big]}.
\end{equation}

The next part of the proof depends on the assumptions we make about
$w(\bt)$,
$ P_{\A_{\bt}}(\kappa) $,
and $  \|g\|^2_{\L_1(\bt)} $ and how they enable us to upper bound
\begin{equation*}
  \E\Bigg[
  \sum_{\bt \in T_{K-1}} w(\bt)P^{-1}_{\A_{\bt}}(\kappa) \|g\|^2_{\L_1(\bt)}
  \Bigg].
\end{equation*}

\textbf{For Lemma~\ref{lmm:training}:}
In this case, we do not impose any assumptions on $w(\bt)$.
We can use the fact that
$\sum_{\bt \in T_{K-1}} w(\bt) = 1$ and
$\|g\|^2_{\L_1(\bt)} \leq  \|g\|^2_{\L_1} $
for all $\bt\in T_{K-1}$
to get
\begin{align*}
  \E\Bigg[\sum_{\bt \in T_{K-1}} w(\bt)P^{-1}_{\A_{\bt}}(\kappa)
  \|g\|^2_{\L_1(\bt)}\Bigg]
  &
    \leq
    \|g\|^2_{\L_1}
    \E\Bigg[\max_{\bt\in T_{K-1}}P^{-1}_{\A_{\bt}}(\kappa)
    \sum_{\bt \in T_{K-1}} w(\bt)\Bigg]
  \\
  &
    = \|g\|^2_{\L_1}\E\Big[\max_{\bt\in T_{K-1}}P^{-1}_{\A_{\bt}}(\kappa)\Big]
  \\
  &
    \leq \|g\|^2_{\L_1}\E\Big[\max_{\bt\in [T_K]}P^{-1}_{\A_{\bt}}(\kappa)\Big].
\end{align*}
Plugging this bound into~\eqref{eq:solution},
we obtain the desired inequality in~\eqref{eq:trainlem} on the expected excess
training error, namely,
\[
\E[R_K]
\leq
\frac{\|g\|^2_{\L_1}
  \E\big[\max_{\bt\in [T_K]} P^{-1}_{\A_{\bt}}(\kappa)\big]}
{\kappa K}.
\]

\textbf{For Lemma~\ref{lmm:fasttraining}:}
If we grant
Assumptions~\ref{as:tvbound} and~\ref{as:node},
and take $ g = \mu \in \G $, we can arrive at a stronger bound.
Recall that we also assume that $ \kappa = 1 $ and $ P_{\A_{\bt}}(\kappa) = 1 $.
Since $q>2$,
by two successive applications of H\"older's inequality, we have
\begin{equation}
\label{eq:holder1}
  \sum_{\bt \in T_{K-1}} w(\bt)
  \|\mu\|^2_{\L_1(\bt)}
  \leq
  \Bigg(\sum_{\bt \in T_{K-1}}(w(\bt))^{q/(q-2)}
  \Bigg)^{1-2/q}
  \Bigg(\sum_{\bt \in T_{K-1}}\|\mu\|^q_{\L_1(\bt)}\Bigg)^{2/q},
\end{equation}
  and
\begin{equation}
  \begin{aligned}
  \label{eq:holder2}
    & \E\Bigg[\Bigg(\sum_{\bt \in T_{K-1}} (w(\bt))^{q/(q-2)}
    \Bigg)^{1-2/q}
    \Bigg(\sum_{\bt \in T_{K-1}}\|\mu\|^q_{\L_1(\bt)}\Bigg)^{2/q}\Bigg]
  \\
  &
    \qquad\qquad
    \leq
    \Bigg(\E\Bigg[\sum_{\bt \in T_{K-1}} (w(\bt))^{q/(q-2)}\Bigg]
    \Bigg)^{1-2/q}
    \Bigg(\E\Bigg[\sum_{\bt \in T_{K-1}}\|\mu\|^q_{\L_1(\bt)}\Bigg]\Bigg)^{2/q}.
  \end{aligned}
\end{equation}
Combining the two inequalities \eqref{eq:holder1} and \eqref{eq:holder2},
we obtain
\begin{equation*}
  \E\Bigg[\sum_{\bt \in T_{K-1}} w(\bt)
  \|\mu\|^2_{\L_1(\bt)}\Bigg]
  \leq
  \Bigg(\E\Bigg[\sum_{\bt \in T_{K-1}} ( w(\bt) )^{q/(q-2)}\Bigg]
  \Bigg)^{1-2/q}
  \Bigg(\E\Bigg[\sum_{\bt \in T_{K-1}}\|\mu\|^q_{\L_1(\bt)}\Bigg]\Bigg)^{2/q}.
\end{equation*}
Assumptions~\ref{as:tvbound} and~\ref{as:node} provide further upper bounds,
since
\begin{align*}
  &
    \Bigg(\E\Bigg[\sum_{\bt \in T_{K-1}} (w(\bt))^{q/(q-2)}\Bigg]
    \Bigg)^{1-2/q}
    \Bigg(\E\Bigg[\sum_{\bt \in T_{K-1}}\|\mu\|^q_{\L_1(\bt)}\Bigg]\Bigg)^{2/q}
  \\
  & \leq
    \Bigg(2^{K-1}
    \E\Bigg[
    \Bigg(\max_{\bt \in T_{K-1}} w(\bt)
    \Bigg)^{q/(q-2)}
    \Bigg]\Bigg)^{1-2/q}
    \Bigg(\E\Bigg[\sum_{\bt \in T_{K-1}}\|\mu\|^q_{\L_1(\bt)}\Bigg]\Bigg)^{2/q}
  \\
  & \leq
    2^{(K-1)(1-2/q)}
    \Bigg(\E\Bigg[
    \Bigg(\max_{\bt \in T_{K-1}} w(\bt)
    \Bigg)^{\nu}
    \Bigg]\Bigg)^{1/\nu}
    \Bigg(\E\Bigg[\sum_{\bt \in T_{K-1}}\|\mu\|^q_{\L_1(\bt)}\Bigg]\Bigg)^{2/q} \\
    & \leq \frac{AV^2}{4^{(K-1)/q}}.
\end{align*}
Plugging this bound into \eqref{eq:solution},
we obtain the desired inequality \eqref{eq:fasttrainlem} on the expected excess
training error, namely,
$
\E[R_K] \leq \frac{AV^2}{4^{(K-1)/q}}
$.
\end{proof}

\begin{proof}[Proof of Theorems~\ref{thm:oracle} and~\ref{thm:fastoracle}]
We begin by splitting the MSE
(averaging only with respect to the joint distribution of
$\{\A_{\bt}: \bt \in [T_k] \}$)
into two terms,
$\E_{T_k}\big[\|\mu-\muhat( T_{k}) \|^{2}\big] = E_{1} + E_{2}$,
where
\begin{align}
  \label{eq:eterms}
  & E_{1}
    =
    \E_{T_k}\big[
    \|\mu-\muhat( T_{k}) \|^{2}\big]
    - 2\big(\E_{T_K}\big[\|\by-\bmuhat( T_{k}) \|^{2}_{n}\big]
    - \|\by-\bmu\|^{2}_{n}\big)
    - \alpha(n, k) - \beta(n)
  \\&
  E_{2}
  =
  2\big(
  \E_{T_k}\big[
  \|\by-\bmuhat( T_{k}) \|^{2}_{n}\big]
  - \|\by-\bmu\|^{2}_{n}\big)
  + \alpha(n, k) + \beta(n),
\end{align}

and where $\alpha(n, k)$ and $\beta(n)$
are positive sequences that will be specified later.

To bound $ \E[E_1] $,
we split our analysis into two cases based on the observed data $y_i$.
Accordingly, we have
\begin{equation}
  \label{eq:esplit}
  \E[E_1]
  =
  \E[E_1\Indicator(\forall i: |y_i|\leq B)]
  +
  \E[E_1\Indicator(\exists i: |y_i|>B)], \quad B \geq 0.
\end{equation}

\subsubsection*{Bounded term}

We start by looking at the first term on the right hand side
of~\eqref{eq:esplit}.

Proceeding, we introduce a few useful concepts and definitions for studying
data-dependent partitions,
due to \citet{nobel1996histogram}. Let
\begin{equation*}
  \Lambda_{n, k}
  =
  \big\{ \mathcal{P}(\{(\tilde y_1, \tilde \bx_1^{\Trans}), \dots,
  (\tilde y_n, \tilde \bx_n^{\Trans})\}) :
  (\tilde y_i, \tilde \bx_i^{\Trans}) \in \mathbb{R}^{1+p} \big\}
\end{equation*}
be the family of all achievable partitions $ \mathcal{P} $
by growing a depth $ k $ oblique decision tree on
$ n $ data points with split boundaries of the form $ \bx^{\Trans}\ba = b $,
where  $ \|\ba\|_{\ell_0} \leq d $.
In particular, note that $ \Lambda_{n, k} $
contains all data-dependent partitions.
We also define
\begin{equation*}
  M(\Lambda_{n, k})
  =
  \max\{\, |\mathcal{P}| : \mathcal{P} \in \Lambda_{n, k}\}
\end{equation*}
to be the maximum number of terminal nodes
among all partitions in $ \Lambda_{n, k} $.
Note that $ M(\Lambda_{n, k}) \leq 2^k $
(this statement does not rely on the specific algorithm used to
grow a depth $k$ oblique tree,
as long as the tree generates a partition of $\X $ at each level).
Given a set $ \bz^n = \{\bz_1, \bz_2, \dots, \bz_n \} \subset \mathbb{R}^p $,
define $ \Gamma(\bz^n, \Lambda_{n, k}) $ to be the number of distinct partitions
of $ \bz^n $ induced by elements of $ \Lambda_{n, k} $, that is, the number of
different partitions $ \{\bz^n \cap A : A \in \mathcal{P} \} $, for $
\mathcal{P} \in \Lambda_{n, k} $.
The partitioning number $ \Gamma_{n, k}(\Lambda_{n, k}) $ is defined by
\[
\Gamma_{n, k}(\Lambda_{n, k})
= \max\{\, \Gamma(\bz^n, \Lambda_{n, k})
: \bz_1, \bz_2, \dots, \bz_n \in \mathbb{R}^p \},
\]
i.e., the maximum number of different partitions of any $n$ point set that can
be induced by members of $ \Lambda_{n, k} $.
Finally, let $ \F_{n, k}(R) $
denote the collection of all functions (bounded by $ R $)
that output an element of
$ \text{span}(\mathcal{H})$
on each region from a partition
$ \mathcal{P} \in \Lambda_{n, k} $.

We can deduce that the partitioning number
is bounded by
\begin{equation*}
  \Gamma_{n, k}(\Lambda_{n, k})
  \leq
  \Bigg(\binom{p}{d}n^d\Bigg)^{2^k}
  \leq \Bigg(\Bigg(\frac{ep}{d}\Bigg)^d n^d\Bigg)^{2^k}
  = \Bigg(\frac{e n p}{d}\Bigg)^{d2^k}.
\end{equation*}
The bound on $\Gamma_{n, k}$ follows from the maximum number of ways in which
$n$ data points can be split by a hyperplane in $d$ dimensions.
The $\binom{p}{d}$ factor accounts for the number of ways in which
a $d$-dimensional hyperplane can be constructed in a $p$-dimensional space.
Note that this bound is not derived from the specific algorithm used to select
the splitting hyperplanes; it is purely
combinatorial.

Then, by slightly modifying the calculations in
\citet[p. 240]{gyorfi2002distribution}
and combining them with
\citet[Lemma 13.1, Theorem 9.4]{gyorfi2002distribution},
we have the following bound for the covering number
$ \N(r, \F_{n, k}(R), \mathscr{L}_1(\mathbb{P}_{\bx^n})) $
of $ \F_{n, k}(R) $
by balls of radius $ r > 0 $
in $ \mathscr{L}_1(\mathbb{P}_{\bx^n}) $
with respect to the empirical discrete measure
$ \mathbb{P}_{\bx^n} $
on
$ \bx^n = \{\bx_1, \bx_2, \dots, \bx_n \} \subset \mathbb{R}^p $:
\begin{equation}
  \label{eq:covering_number}
\begin{aligned}
  \N \Bigg(\frac{\beta(n)}{40 R},
  \F_{n, k}(R),
  \mathscr{L}_1(\mathbb{P}_{\bx^n})\Bigg)
  &
    \leq \Gamma_{n, k}(\Lambda_{n, k})
    \Bigg(3 \Bigg(
    \frac{6eR}{\frac{\beta(n)}{40R}}
    \Bigg)^{2\text{VC}(\calH)}\Bigg)^{2^{k}}
  \\
  &
    \leq
    \Bigg(\Bigg(\frac{e n p}{d}\Bigg)^d\Bigg)^{2^k}
    \Bigg(3 \Bigg(\frac{240 e R^{2}}{\beta(n)}\Bigg)^{2\text{VC}(\calH)}\Bigg)^{2^{k}}
  \\
  & = \Bigg(3\Bigg(\frac{e n p}{d}\Bigg)^d\Bigg)^{2^k}
    \Bigg(\frac{240 e R^{2}}{\beta(n)}\Bigg)^{\text{VC}(\calH) 2^{k+1}},
\end{aligned}
\end{equation}
where we use $\text{VC}(\calH)$ to denote the VC dimension of
$\text{span}(\calH)$.
According to~\eqref{eq:uniform},
we know that the regression function is uniformly bounded,
$ \|\mu\|_{\infty} \leq M^{\prime} $. Let $ R = QB $.
We assume, without loss of generality, that
$R \geq M^{\prime}$
so that
$ \|\mu\|_{\infty} \leq R $
and $ ||\muhat( T_k)\|_{\infty} \leq R $ almost surely,
if
$ \max_{1 \leq i \leq n}|y_i| \leq B $.
By \citet[Theorem 11.4]{gyorfi2002distribution},
with $\epsilon = 1/2$ (in their notation),
\begin{equation}
  \begin{aligned}
    \label{eq:covrbound}
    &
      \Prob
      \Big(
      \exists f \in \F_{n,k}(R) :
      \|\mu - f\|^{2}
      \geq
      2(\|\by-\bff\|_n^{2} - \|\by-\bmu\|_{n}^{2})
      + \alpha(n, k) + \beta(n), \;
      \forall i : |y_i|\leq B
      \Big)
    \\
    &
      \qquad\qquad
      \leq
      14
      \sup_{\bx^{n}}
      \N \Bigg(\frac{\beta(n)}{40R}, \F_{n,k}(R), \mathscr{L}_{1}(\Prob_{\bx^{n}}) \Bigg)
      \exp \Bigg( - \frac{\alpha(n, k) n}{2568 R^{4}} \Bigg).
  \end{aligned}
\end{equation}
Then, we have the following probability concentration
\begin{align}
  \label{eq:prob_bound}
  &  \Prob
    \big(
    \E_{T_K}\big[ \|\mu - \muhat( T_{k})\|\big]
    \geq
    2(\E_{T_K}\big[\|\by-\bmuhat( T_{k})\|_n^{2}\big] - \|\by-\bmu\|_{n}^{2})
    + \alpha(n, k) + \beta(n),\; \forall i : |y_i|\leq B
    \big)
    \nonumber
  \\
  & \qquad \qquad \qquad
    \leq
    14
    \sup_{\bx^{n}}
    \N \Bigg(\frac{\beta(n)}{40R}, \F_{n,k}(R), \mathscr{L}_{1}(\Prob_{\bx^{n}}) \Bigg)
    \exp \Bigg( - \frac{\alpha(n, k) n}{2568 R^{4}} \Bigg).
\end{align}

This inequality follows from the fact that, on the event
$\{\forall i : |y_i|\leq B\} $, if
\begin{equation*}
  \E_{T_K}\big[ \|\mu - \muhat( T_{k})\|^{2}
  -2\|\by-\bmuhat( T_{k})\|_n^{2}\big] \geq  -2\|\by-\bmu\|_{n}^{2}
  + \alpha(n, k) + \beta(n)
\end{equation*}
holds, then there exists a realization
$\muhat(T^{\prime}_{k})\in \F_{n,k}(R)$
such that
\begin{equation*}
  \|\mu - \muhat(T^{\prime}_{k})\|^{2}
  -2\|\by-\bmuhat(T^{\prime}_{k})\|_n^{2} \geq  -2\|\by-\bmu\|_{n}^{2}
  + \alpha(n, k) + \beta(n),
\end{equation*}
and hence
\begin{align*}
  &
    \Prob
    \Big(\E_{T_K}\big[ \|\mu - \muhat( T_{k})\|\big]
    \geq
    2(\E_{T_K}\big[\|\by-\bmuhat( T_{k})\|_n^{2}\big] - \|\by-\bmu\|_{n}^{2})
    + \alpha(n, k) + \beta(n), \; \forall i : |y_i|\leq B
    \Big)
  \\
  &
    \qquad
    \leq
    \Prob
    \Big(
    \exists f \in \F_{n,k}(R) :
    \|\mu - f\|^{2}
    \geq
    2(\|\by-\bff\|_n^{2} - \|\by-\bmu\|_{n}^{2})
    + \alpha(n, k) + \beta(n), \; \forall i : |y_i|\leq B
    \Big).
\end{align*}
We can now plug in the result of~\eqref{eq:covering_number}
into~\eqref{eq:prob_bound} to obtain
\begin{equation}
  \label{eq:gyorfibound}
  \Prob(E_{1} \geq 0, \; \forall i : |y_i|\leq B)
  \leq
  14
   \Bigg(3\Bigg(\frac{e n p}{d}\Bigg)^d\Bigg)^{2^k}
  \Bigg(\frac{240 e R^2}{\beta(n)}\Bigg)^{\text{VC}(\calH) 2^{k+1}}
  \exp\Bigg( - \frac{\alpha(n, k) n}{2568 R^4} \Bigg).
\end{equation}

We choose
\begin{align*}
  \alpha(n, k)
  &=
  \frac{2568 R^4\Big(2^{k}d\log(e n p/d) + 2^{k}\log (3) +
  \text{VC}(\mathcal{H})2^{k+1}\log ( \frac{240 e R^2}{\beta(n)} )
    + \log (14n^2) \Big)}{n}
  \\
  \beta(n) &= \frac{240 e R^2}{n^2}
\end{align*}
so that
$\Prob(E_{1} \geq 0, \; \forall i : |y_i|\leq B) \leq 1/n^2$.
Thus,
\begin{equation*}
  E_{1}\Indicator(\forall i: |y_i|\leq B)
  \leq
  \big(\E_{T_K}\big[\|\mu-\muhat( T_{k}) \|^{2}\big]
  + 2\|\by-\bmu\|^{2}_{n}\big)
  \Indicator(\forall i: |y_i|\leq B)
  \leq
  12R^2,
\end{equation*}
and so we have
\begin{equation}
  \label{eq:e1bounded}
  \E\big[E_1\Indicator(\forall i: |y_i|\leq B)\big]
  \leq 12 R^2
  \Prob(E_{1}\geq 0, \; \forall i : |y_i|\leq B)
  \leq
  \frac{12 R^2}{n^2} = \frac{12 Q^2B^2}{n^2}.
\end{equation}

\subsubsection*{Unbounded term}

We now look at the second term on the right hand side of~\eqref{eq:esplit}.
Because we have
$ \|\muhat(T_k)\|_{\infty}
\leq
Q \cdot \sqrt{\max_{1 \leq i\leq n}
  \frac{1}{i}\sum_{\ell=1}^i y^2_{\ell}}$
almost surely,
we can bound
\begin{equation*}
  \E\big[
  \|\mu - \muhat( T_{k})\|^{2}\Indicator(\exists i: |y_i|>B)
  \big]
  \leq
  (Q+1)^2\E\Big[
  \max_{1\leq i\leq n}\max\big\{y^2, y_i^2\big\}\Indicator(\exists i: |y_i|>B)
  \Big].
\end{equation*}
Using the fact that the sum of non-negative variables upper bounds their maximum,
and the
exponential concentration
of the conditional distribution of $ y $ given $ \bx $
(Assumption~\ref{as:dgp}) together with a union bound,
we can then apply Cauchy-Schwarz to obtain
\begin{equation*}
  \E\big[
  \|\mu - \muhat( T_{k})\|^{2}\Indicator(\exists i: |y_i|>B)
  \big]
  \leq
  (Q+1)^2
  \textstyle\sqrt{(n+1)\E[y^4]}
  \textstyle\sqrt{nc_1\exp(-c_2(B-M)^{\gamma})}.
\end{equation*}
Setting
$B=B_n = M + \big((6/c_2)\log(n+1)\big)^{1/\gamma} \geq M^{\prime} $,
we have that
\begin{equation}
  \label{eq:tailbound}
  \E\big[\|\mu-\muhat( T_{k})\|^{2}\Indicator(\exists i: |y_i|>B)\big]
  \leq
  \frac{(Q+1)^2{\textstyle\sqrt{c_1\E\big[y^4\big]}}}{n^2}.
\end{equation}
Thus combining~\eqref{eq:e1bounded} and~\eqref{eq:tailbound},
we have
\begin{equation}
\label{eq:tailandbounded}
\begin{aligned}
  \E[E_1]
  & =
    \E[E_1\Indicator(\forall i: |y_i|\leq B)]
    +
    \E[E_1\Indicator(\exists i: |y_i|>B)]
  \\
  & \leq
    \frac{12 Q^2 B^2}{n^2}
    + \frac{(Q+1)^2{\textstyle\sqrt{c_1\E\big[y^4\big]}}}{n^2}
    = O\Bigg(\frac{\log^{2/\gamma}(n)}{n^2}\Bigg).
\end{aligned}
\end{equation}

Next, we turn our attention to $ \E[E_2] $. Since
\[
\E\big[
  \|\by-\bmuhat( T_{k}) \|^{2}_{n}
  - \|\by-\bmu\|^{2}_{n}\big] = \|\mu-g\|^2 + \E\big[
  \|\by-\bmuhat( T_{k}) \|^{2}_{n}
  - \|\by-\bg\|^{2}_{n}\big],
\]
it follows that
\begin{equation}
  \label{eq:e2eq}
  \E[E_2] = 2\|\mu-g\|^2
  + 2\E\big[
  \|\by-\bmuhat( T_{k}) \|^{2}_{n}
  - \|\by-\bg\|^{2}_{n}\big] + \alpha(n, k) + \beta(n).
\end{equation}
Finally,
combining the bounds~\eqref{eq:tailandbounded}
and~\eqref{eq:e2eq} and simplifying $\alpha(n, k)$ and $\beta(n)$,
\begin{equation}
  \begin{aligned}
    \label{eq:ebound}
    &
      \E\big[\|\mu-\muhat( T_{K}) \|^{2}\big]
    \\
    &
      \quad \leq
      2\|\mu-g\|^2
      +
      2\E\big[\|\by-\bmuhat( T_{K}) \|^{2}_{n}
      - \|\by-\bg\|^{2}_{n}\big]
      + C\frac{2^{K}(d + \text{VC}(\calH))\log(np/d)\log^{4/\gamma}(n)}{n},
  \end{aligned}
\end{equation}
for some positive constant $ C=C(c_1,c_2,\gamma, M, Q)$.

\subsubsection*{Pruned tree}

We now consider the pruned tree, $T_{\text{opt}}$.
Let
$\E_{T_{\text{opt}}}\big[\|\mu-\muhat(T_{\text{opt}}) \|^{2}\big]
= E^{\prime}_{1} + E^{\prime}_{2}$,
where
\begin{align*}
    E^{\prime}_{1}
  &
    =
    \E_{T_{\text{opt}}}\big[
    \|\mu-\muhat( T_{\text{opt}}) \|^{2}\big]
    - 2(\E_{T_{\text{opt}}}\big[\|\by-\bmuhat( T_{\text{opt}}) \|^{2}_{n}\big]
    - \|\by-\bmu\|^{2}_{n})
    - 2\lambda |T_{\text{opt}}|
  \\
    E^{\prime}_{2}
  &
    =
    2(
    \E_{T_{\text{opt}}}\big[
    \|\by-\bmuhat( T_{\text{opt}}) \|^{2}_{n}\big]
    - \|\by-\bmu\|^{2}_{n})
    + 2\lambda |T_{\text{opt}}|.
\end{align*}
Note that, for each $ k = 1, 2, \dots, n-1 $,
\begin{equation*}
  \|\by - \bmuhat( T_{\text{opt}})\|^2_n
  + \lambda |T_{\text{opt}}|
  \leq
  \|\by - \bmuhat( T_k)\|^2_n + \lambda 2^k,
\end{equation*}
and hence, for each $ k \geq 1 $,
\begin{equation} \label{eq:e2prime}
\E[E^{\prime}_{2}] \leq 2\|\mu-g\|^2 + 2\E\big[
  \|\by-\bmuhat( T_{k}) \|^{2}_{n}
  - \|\by-\bg\|^{2}_{n}\big] + \lambda 2^{k+1}.
\end{equation}
Choose
$ \lambda = \lambda_{n} $
such that
$ \alpha(n, k) + \beta(n) \leq \lambda_{n} 2^{k+1} $.
This implies that
$ \lambda_{n} \gtrsim \frac{(d+\text{VC}(\calH))\log(np/d)\log^{4/\gamma}(n)}{n} $.
For each realization of $T_{\text{opt}}$,
there exists $k$ such that $|T_{\text{opt}}| \geq 2^k$.
By a union bound and the result established in~\eqref{eq:gyorfibound},
we have
\begin{align*}
  P(E_1^{\prime} \geq 0)
  & \leq
    \Prob(\E_{T_{\text{opt}}}\big[\|\mu - \muhat( T_{\text{opt}}) \|^2\big]
    \geq 2(\E_{T_{\text{opt}}}\big[\|\by-\bmuhat( T_{\text{opt}})\|^2_n\big]
    - \|\by-\bmu\|^2_n)
    + 2\lambda_{n} |T_{\text{opt}}|)
  \\
  & \leq
    \sum_{1 \leq k \leq n-1} \Prob
    (
    \exists f \in \F_{n,k}(R) :
    \|\mu - f\|^{2}
    \geq
    2(\|\by-\bff\|_n^{2} - \|\by-\bmu\|_{n}^{2})
    + \lambda_{n} 2^{k+1}
    )
  \\
  & \leq
    \sum_{1 \leq k \leq n-1}
    \Prob
    (
    \exists f \in \F_{n,k}(R) :
    \|\mu - f\|^{2}
    \geq
    2(\|\by-\bff\|_n^{2} - \|\by-\bmu\|_{n}^{2})
    + \alpha(n, k) + \beta(n)
    )
  \\
  &
    \leq \sum_{1 \leq k \leq n-1}n^{-2}
    \leq 1/n.
\end{align*}
Once again, we split the expectation, $\E[E^{\prime}_1]$ into two cases,
as in~\eqref{eq:esplit},
and bound each case separately.
The argument is identical to that for the un-pruned tree
so we omit details here.
Combining this bound on $\E[E^{\prime}_1]$ with~\eqref{eq:e2prime}
gives as an analogous result to~\eqref{eq:ebound}, namely,
for all $K\geq 1$,
\begin{equation}
  \begin{aligned}
    \label{eq:eprimebound}
    &
      \E\big[\|\mu-\muhat( T_{\text{opt}}) \|^{2}\big]
    \\
    &
      \quad \leq
      2\|\mu-g\|^2
      +
      2\E\big[\|\by-\bmuhat( T_{K}) \|^{2}_{n}
      - \|\by-\bg\|^{2}_{n}\big]
      + C\frac{2^{K}(p + \text{VC}(\calH))\log^{1+4/\gamma}(n)}{n},
  \end{aligned}
\end{equation}
for some positive constant $ C=C(c_1,c_2,\gamma, M, Q)$.

The next part of the proof entails bounding
$ \E\big[\|\by-\bmuhat( T_{K}) \|^{2}_{n} - \|\by-\bg\|^{2}_{n}\big] $,
depending on the assumptions we make.
Note that for the constant output
$ \hat{y}_{\bt}(\bx) \equiv \overline y_{\bt} $,
we have
$Q=1$ and $\text{VC}(\calH)=1$.

\textbf{For Theorem~\ref{thm:oracle}:}
We bound
$ \E\big[\|\by-\bmuhat( T_{K}) \|^{2}_{n} - \|\by-\bg\|^{2}_{n}\big] $
using
Lemma~\ref{lmm:training}.
The inequality~\eqref{eq:oracle_unprune}
follows directly from \eqref{eq:ebound}
and the inequality~\eqref{eq:oracle_prune} follows directly from~\eqref{eq:eprimebound}.

\textbf{For Theorem~\ref{thm:fastoracle}:}
Taking $g=\mu \in \G $ and $d=p$, we bound
$ \E\big[\|\by-\bmuhat( T_{K}) \|^{2}_{n} - \|\by-\bg\|^{2}_{n}\big] $
using
Lemma~\ref{lmm:fasttraining}.
The inequality~\eqref{eq:rate_unprune} follows directly
from~\eqref{eq:ebound}.
To show~\eqref{eq:rate_prune}, we use~\eqref{eq:rate_unprune} and
\begin{align}
&  \inf_{K\geq 1}\Bigg\{
  \frac{2AV^2}{4^{(K-1)/q}}
  +
  C\frac{2^{K+1} p\log^{4/\gamma+1}(n)}{n}
  \Bigg\}
  \\ & \qquad =   2(2+q)
  \Bigg(\frac{AV^2}{q} \Bigg)^{q/(2+q)}
  \Bigg(\frac{Cp \log^{4/\gamma+1}(n)}{n}\Bigg)^{2/(2+q)}.
\end{align}
This completes the proof of both Theorem~\ref{thm:oracle}
and Theorem~\ref{thm:fastoracle}.
\end{proof}

\begin{proof}[Proof of Corollary~\ref{cor:fixed_d}]

Because our risk bounds allow for model misspecification, one can easily
establish consistency of $ \widehat{\mu}( T_{K}) $, even when $ \mu
\in \F \setminus \G$.
Recall that $\F = \text{cl}(\G)$, that is,
\begin{equation*}
  \F = \Bigg\{
  f(\bx) = \sum_{k=1}^M f_{k}(\ba_k^{\Trans}\bx),\; \ba_k \in \R^p, \;
  f_k: \R \mapsto \R \Bigg\}.
\end{equation*}
Importantly, $\F$ includes functions whose $\L_1$ norm may be infinite.
Consider such a function $\mu$ that belongs to $\F$ but not to $\G$.
Furthermore, grant Assumptions~\ref{as:unifprob} and~\ref{as:dgp},
which entail $ \mu \in \mathscr{L}_{\infty}(\mathbb{R}^p) $.
Let $\G' \subset \G$ denote the set of
all single-hidden layer feed-forward neural networks with activation function
that is non-constant and of bounded variation (and hence bounded). Then by
\citet[Theorem 1]{hornik1991approximation},  $ \G' $ is dense in $
\mathscr{L}_{\infty}(\mathbb{R}^p) \subset \mathscr{L}_2(\mathbb{P}_{\bx}) $.
Therefore, we can choose a sequence $\{g_n\} \subset \G' $, where each component
function $g_{nk}$ is bounded, non-constant, and of bounded variation, such that
$\lim_{n\rightarrow \infty}\|\mu-g_n\| = 0 $
and $ \|g_n\|_{\L_1} < \infty $ for
each $ n$. Define a subsequence $\{ g_{a_n}\} $ by $ a_n = \max\Big\{ m \leq n :
\|g_m\|_{\L_1} \leq D\sqrt{K_n/\log(n+1)} \Big\} $, where $ D $ is a positive
constant large enough so that $ \|g_1\|_{\L_1} \leq D\sqrt{K_n/\log(n+1)} $ for
all $ n $. Then, by construction, we have $\|\mu-g_{a_n}\| \rightarrow 0 $ and $
\|g_{a_n}\|_{\L_1} = o(\sqrt{K_n}) $ as $ n \rightarrow \infty $. Finally,
according to \eqref{eq:oracle_unprune} (and
similarly~\eqref{eq:oracle_prune}),
since $ \{g_{a_n}\} \subset \F $, we have
$ \lim_{n\rightarrow \infty}\mathbb{E}\big[\|\mu - \widehat{\mu}(
T_{K})\|^2\big] = 0$.

An analogous argument holds for the pruned tree $T_{\text{opt}}$.
\end{proof}

\begin{proof}[Proof of Corollary~\ref{cor:oracle_consistency}]
  The proof follows directly from the assumptions and Theorem~\ref{thm:oracle}.
\end{proof}

\begin{proof}[Proof of Theorem~\ref{thm:orf_oracle}]

Since we assume the subsample selection is independent of the splitting direction subset
selection at each node, we have the following decomposition of the law of the
process that governs each tree in the forest:
\[\Pi_{\Theta} = \Pi_{K} \times \Pi_{\I},\]
where $\I \subset \{1, \ldots, n\}$ is the set of indices of the subsampled data set of size $N$.

\subsubsection*{Part 1: Training error bound}
By Jensen's inequality,
\begin{equation*}
  \E_{\Pi_{\Theta}} \big[\|\mu -\muhat(\boldsymbol{\Theta})\|^2\big]
  \leq
  \E_{\Pi_{\Theta}} \big[\|\mu -\muhat(T_{K}(\Theta))\|^2\big].
\end{equation*}

Additionally, by the law of total expectation,
\begin{equation*}
\E_{\Pi_{\Theta}}\big[\|\mu-\muhat(T_K(\Theta))\|^2 \big] =
\E_{\I}\big[\E_{\Pi_K}\big[\|\mu-\muhat(T_K(\Theta))\|^2 \mid \I
\big]\big].
\end{equation*}

We can prove a training error bound analogous to that of
Lemma~\ref{lmm:fasttraining} by considering the modified definitions of excess
training error.
Define excess training error at each node conditional on the subsampled data as
\begin{equation*}
  R^{\I}_{K}(\bt)
  = \|\by-\overline{y}_{\bt}\|_{\bt, \I}^{2} - \|\by-\bg\|_{\bt}^{2},
\end{equation*}

and the excess training error of the tree as
\begin{equation*}
  R^{\I}_{K}
  = \norm{\by-\overline{y}}_{\I}^{2} - \norm{\by-\bg}_{\I}^{2}.
\end{equation*}

Since we do the subset selection independently at each node, any terminal node
$\bt$ of $T_{K-1}$ is independent of $\Pi_{K}$, conditional on $\Pi_{K-1}$.
We can then apply the law of iterated expectation to the conditional training
error, just as in the proof of Lemma~\ref{lmm:training}
and the bound follows directly.

\subsubsection*{Part 2: Oracle inequality}
The second part of this proof is analogous to the proof
Theorem~\ref{thm:oracle} where the averaging over the data set is replaced by
averaging over the subsampled data.

This completes the proof.
\end{proof}

\section{Technical Lemmas}\label{sec:techlem}

In this section, we present some technical lemmas that aid in the proof of our
main results and may also be of independent interest.

\subsection{Impurity Bound}
~\label{sec:impurity_bound}
Our first lemma establishes an important connection between the decrease in
impurity and the empirical node-wide excess risk.

\begin{lemma}[Impurity bound]~\label{lem:impurity_bound}
  Define $R_{K-1}(\bt) = \|\by-\hat{y}_{\bt}\|^{2}_{\bt} - \|\by-\bg\|^{2}_{\bt}$.
  Let $\bt$ be a terminal node of $T_{K-1}$, and
  assume $R_{K-1}(\bt) > 0$.
  Then, if $ g \in \G $,
  \begin{equation}
    \max_{(b, \ba) \in \mathbb{R}^{1+p}}\widehat{\Delta}(b, \ba, \bt)
    \geq
    \frac{w(\bt)R^2_{K-1}(\bt)}{\|g\|^2_{\L_1(\bt)}}.
  \end{equation}
\end{lemma}

\begin{proof}[Proof of Lemma~\ref{lem:impurity_bound}]
Assume that $g \in \G$,
$ g(\bx) = \sum_{k=1}^Mg_k(\ba_k^{\Trans}\bx) $, and that
$g(\bx_i)$ is not constant across $ \bx_i \in \bt $, the
result being trivial otherwise.
We use $g_{k}^{\prime}$ to denote the divided difference of
$g_{k}$ of successive ordered datapoints in the
$\ba_k$ direction in node $\bt$.
That is, if the data
$ \{(y_i, \bx_i^{\Trans}): \bx_i \in \bt \} $
is re-indexed so that
$\ba_k^{\Trans}\bx_{1}
\leq \ba_k^{\Trans}\bx_{2}
\leq \cdots
\leq \ba_k^{\Trans}\bx_{n(\bt)} $, then
\begin{equation}
  \label{eq:gderiv}
  g_{k}^{\prime}(b) =
  \frac{
  g_{k}(\ba_k^{\Trans}\bx_{i+1})
  - g_{k}(\ba_k^{\Trans}\bx_{i})
  }{
  \ba_k^{\Trans}\bx_{i+1}
  - \ba_k^{\Trans}\bx_{i}},
  \quad \text{for} \
  \ba_k^{\Trans}\bx_{i} \leq b < \ba_k^{\Trans}\bx_{i+1}
  \quad \text{and} \
  i = 1, \ldots, n(\bt)-1,
\end{equation}
where $ g_{k}^{\prime}(b) = 0 $ if $ b= \ba_k^{\Trans}\bx_{i} = \ba_k^{\Trans}\bx_{i+1} $.
Let
\begin{equation}
  \label{eq:rd}
  \frac{\diff\Pi(b, \ba_k)}{\diff(b, \ba_k)}
  =
  \frac{ | g_k^{\prime}(b) |
  \sqrt{\Prob(\bt_L) \Prob(\bt_{R})}
  }{
  \sum_{k^{\prime}=1}^M\int
  | g_{k^{\prime}}^{\prime}(b^{\prime}) |
  \sqrt{\Prob(\bt_L^{\prime}) \Prob(\bt_{R}^{\prime})} \diff b^{\prime} }
\end{equation}
denote the Radon-Nikodym derivative
(with respect to the Lebesgue measure and counting measure)
of a probability measure on $ (b, \ba) $ after splitting node
$\bt$ at the decision boundary $ \ba_k^{\Trans}\bx  = b $.
Here
$\bt_L=\bt_L(b, \ba_k)$ and $\bt_R = \bt_R(b, \ba_k)$
are the child nodes of $\bt$ after splitting at $ \ba_k^\Trans \bx = b$,
and
$\Prob(\bt_L) = n(\bt_L)/n(\bt)$ and $\Prob(\bt_R) = n(\bt_R)/n(\bt)$
are the proportions of observations in node
$\bt$ that is in $\bt_L$ and $\bt_R$, respectively.
Similarly,
$\bt^{\prime}_L=\bt_L^{\prime}(b^{\prime}, \ba_{k^{\prime}})$
and $\bt^{\prime}_R = \bt^{\prime}_R(b^{\prime}, \ba_{k^{\prime}})$
are the child nodes of $\bt$ after splitting at
$ \ba_{k^{\prime}}^\Trans \bx = b^{\prime}$.
Additionally, define
\begin{equation*}
  \tilde{\psi}_{\bt}(\bx)
  = \frac{\Indicator(\bx \in \bt_L)\Prob(\bt_R)
    - \Indicator(\bx \in \bt_R)\Prob(\bt_L)}
  {\sqrt{\Prob(\bt_L)\Prob(\bt_R)}}
  = \textstyle\sqrt{w(\bt)}\psi_{\bt}(\bx).
\end{equation*}
Note that
$  \{ \tilde{\psi}_{\bt}: \bt\in [T_K]\} $
is an orthonormal dictionary with respect to the node-wise inner product,
$ \langle \cdot, \cdot \rangle_{\bt} $.

Because a maximum is larger than an average,
$
\max_{(b, \ba) \in \mathbb{R}^{1+p}}\widehat{\Delta}(b, \ba, \bt)
\geq
\int \widehat{\Delta}(b, \ba_k, \bt) \diff \Pi(b, \ba_k).
$
Then, using the identity from~\eqref{eq:impuritydec}
and the fact that the decision stump
$ \psi_{\bt} $ belongs to $\Psi_{\bt}^{\perp}$ (see~\eqref{eq:stumps_norm}),
we have
\begin{equation}
  \label{eq:deltaboundinitial}
  \max_{(b, \ba) \in \mathbb{R}^{1+p}}\widehat{\Delta}(b, \ba, \bt)
  \geq
  \int \sum_{\psi \in \Psi_{\bt}^{\perp}}
  |\langle \by, \bpsi \rangle_n|^2
  \diff \Pi(b, \ba_k)
  \geq
  \int |\langle \by, \bpsi_{\bt} \rangle_n|^2
  \diff \Pi(b, \ba_k).
\end{equation}
By the definition of $\tilde{\psi}_{\bt}$
and Jensen's inequality,
\begin{equation}
  \label{eq:deltabound}
    \int |\langle \by, \bpsi_{\bt} \rangle_n|^2
    \diff \Pi(b, \ba_k)
     =
    w(\bt)
    \int | \langle \by, \tilde{\bpsi}_{\bt}\rangle_{\bt} |^{2}
    \diff \Pi(b, \ba_k)
   \geq
    w(\bt)
    \Bigg(
    \int |
    \langle \by, \tilde{\bpsi}_{\bt}\rangle_{\bt} |
    \diff \Pi(b, \ba_k)
    \Bigg)^{2}.
\end{equation}

Our next task will be to lower bound the expectation
$ \int | \langle \by, \tilde{\bpsi}_{\bt}\rangle_{\bt} | \diff \Pi(b, \ba_k) $.
First note the following identity:
\begin{equation*}
  \Indicator(\bx\in \bt_L)\Prob(\bt_{R})
  -
  \Indicator(\bx\in \bt_{R})\Prob(\bt_L)
  =
  -
  (\Indicator (\bx^{\Trans}\ba>b) - \Prob(\bt_{R}))
  \Indicator(\bx\in \bt),
\end{equation*}
which means
\begin{equation*}
\textstyle\sqrt{\Prob(\bt_L) \Prob(\bt_{R})}
  \langle \by, \tilde{\bpsi}_{\bt} \rangle_{\bt}
  = \textstyle\sqrt{\Prob(\bt_L) \Prob(\bt_{R})}
  \langle \by-\hat{y}_{\bt}, \tilde{\bpsi}_{\bt} \rangle_{\bt}
  = -\langle \by-\hat{y}_{\bt}, \Indicator (\bx^{\Trans}\ba>b)\rangle_{\bt}.
\end{equation*}
Using this identity together with the
empirical measure (defined in~\eqref{eq:rd}),
we see that the expectation in~\eqref{eq:deltabound}
is lower bounded by
\begin{equation}
\label{eq:deltaboundintermediate}
\begin{aligned}
  \int |\langle
  \by-\hat{y}_{\bt}, \tilde{\bpsi}_{\bt}
  \rangle_{\bt}| \diff \Pi(b, \ba_k)
  &= \frac{
    \sum_{k=1}^M \int
    |g_{k}^{\prime}(b)|
    |\langle \by-\hat{y}_{\bt}, \Indicator (\ba_k^{\Trans}\bx>b)\rangle_{\bt}|
    \diff b
    }{
    \sum_{k^{\prime}=1}^M\int
    |g_{k^{\prime}}^{\prime}(b^{\prime})|
    \sqrt{\Prob(\bt_L^{\prime}) \Prob(\bt_{R}^{\prime})} \diff b^{\prime}
    }
  \\
  & \geq
    \frac{
    | \langle
    \by-\hat{y}_{\bt},
    \sum_{k=1}^M \int
    g_{k}^{\prime}(b)
    \Indicator (\ba_k^{\Trans}\bx>b)\diff b
    \rangle_{\bt}
    | }{
    \sum_{k^{\prime}=1}^M\int
    |g_{k^{\prime}}^{\prime}(b^\prime)|
    \sqrt{\Prob(\bt_L^{\prime}) \Prob(\bt_{R}^{\prime})}
    \diff b^{\prime} }.
\end{aligned}
\end{equation}
Then, by the definition of $g^{\prime}_k$, we have
$ \sum_{k=1}^M \int
g_{k}^{\prime}(b)
\Indicator (\ba_k^{\Trans}\bx_i>b)\diff b = g(\bx_i) - g(\bx_1) $
for each $ i = 1, 2, \dots, n(\bt) $, and hence
\begin{equation}
  \label{eq:deltaboundfinal}
  \Big\langle
  \by-\hat{y}_{\bt},
  \sum_{k=1}^M \int
  g_{k}^{\prime}(b)
  \Indicator (\ba_k^{\Trans}\bx>b)\diff b
  \Big\rangle_{\bt}
  =
  \langle \by-\hat{y}_{\bt}, \bg\rangle_{\bt}.
\end{equation}

In light
of~\eqref{eq:deltaboundinitial},~\eqref{eq:deltabound},~\eqref{eq:deltaboundintermediate},
and~\eqref{eq:deltaboundfinal}, we obtain
\begin{equation}
  \label{eq:intermediatebound}
  \max_{(b, \ba) \in \mathbb{R}^{1+p}}\widehat{\Delta}(b, \ba, \bt)
  \geq
  \frac{w(\bt)| \langle \by-\hat{y}_{\bt}, \bg\rangle_{\bt} |^{2}
  }{
    ( \sum_{k^{\prime}=1}^M
    \int
    |g_{k^{\prime}}^{\prime}(b^\prime)|
    \sqrt{\Prob(\bt_L^{\prime}) \Prob(\bt_{R}^{\prime})}
    \diff b^{\prime} )^{2} }.
\end{equation}

Next, we derive upper and lower bounds on the
denominator and numerator of~\eqref{eq:intermediatebound}, respectively.
First, we look at the denominator.
Note that for each $k'$,
the integral can be decomposed as follows:
\begin{equation}
  \label{eq:piecewisebound}
  \int
  |g_{k^{\prime}}^{\prime}(b^\prime)|
  \textstyle\sqrt{\Prob(\bt_L^{\prime}) \Prob(\bt_{R}^{\prime})}
  \diff b^{\prime}
    =
    \sum_{i=1}^{n(\bt)-1} \int_{\{b^{\prime}:n(\bt^{\prime}_L)=i\}}
    |g_{k^{\prime}}^{\prime}(b^\prime)|
    \textstyle\sqrt{(i/n(\bt))(1-i/n(\bt))}
    \diff b^{\prime}.
\end{equation}
Then, using the fact that${\textstyle\sqrt{(i/n(\bt))(1-i/n(\bt))}} \leq 1/2$
for $1 \leq i \leq n(\bt)$,
and that the end points of each integral in the sum of~\eqref{eq:piecewisebound}
can be explicitly identified from the definition of $g^{\prime}_{k^{\prime}}$ in~\eqref{eq:gderiv},
\begin{equation}
  \label{eq:denomboundpart}
  \int
  |g_{k^{\prime}}^{\prime}(b^\prime)|
  {\textstyle\sqrt{\Prob(\bt_L^{\prime}) \Prob(\bt_{R}^{\prime})}}
  \diff b^{\prime}
  \leq
  \frac{1}{2}
  \sum_{i=1}^{n(\bt)-1}\int_{\{b^{\prime}:n(\bt^{\prime}_L)=i\}}
  |g_{k^{\prime}}^{\prime}(b^\prime)|
  \diff b^{\prime}
  =
  \frac{1}{2}
  \sum_{i=1}^{n(\bt)-1}
  \int_{\ba_{k^{\prime}}^\Trans\bx_i}^{\ba_{k^{\prime}}^\Trans\bx_{i+1}}
  |g_{k^{\prime}}^{\prime}(b^\prime)|
  \diff b^{\prime}.
\end{equation}
By the definition of $ g^{\prime}_{k^{\prime}}$ as a divided
difference~\eqref{eq:gderiv}
and the definition of total variation,
for each $k^{\prime}$,
\begin{equation}
  \label{eq:denombound}
  \sum_{i=1}^{n(\bt)-1}
  \int_{\ba_{k^{\prime}}^\Trans\bx_i}^{\ba_{k^{\prime}}^\Trans\bx_{i+1}}
  |g_{k^{\prime}}^{\prime}(b^\prime)|
  \diff b^{\prime}
  =
  \sum_{i=1}^{n(\bt)-1}
  |g_{k^{\prime}}(\ba_{k^{\prime}}^\Trans\bx_{i+1})
  - g_{k^{\prime}}(\ba_{k^{\prime}}^\Trans\bx_{i})|
  \leq
  V(g_{k^{\prime}}, \ba_{k^{\prime}}, \bt).
\end{equation}

Combining~\eqref{eq:denomboundpart} and~\eqref{eq:denombound}
and plugging the result into the summation in the
denominator of~\eqref{eq:intermediatebound},
we get
\begin{equation*}
  \sum_{k^{\prime}=1}^M
  \int
  |g_{k^{\prime}}^{\prime}(b^{\prime})|
  {\textstyle\sqrt{\Prob(\bt_L^{\prime}) \Prob(\bt_{R}^{\prime})}}
  \diff b^{\prime}
  \leq
  \frac{1}{2}
  \sum_{k^{\prime}=1}^MV(g_{k^{\prime}}, \ba_{k^{\prime}}, \bt)
  = \frac{1}{2} \|g\|_{\L_1(\bt)}.
\end{equation*}

Next, we lower bound the numerator in~\eqref{eq:intermediatebound}.
Using the Cauchy-Schwarz inequality and the fact that
$\langle \by-\hat{y}_{\bt}, \by\rangle_{\bt}
= \|\by-\hat{y}_{\bt}\|_{\bt}^{2}$,
we obtain
\begin{equation}
  \label{eq:innerprod}
  \langle \by-\hat{y}_{\bt}, \bg \rangle_{\bt}
  =
  \langle \by-\hat{y}_{\bt}, \by\rangle_{\bt}
  - \langle \by-\hat{y}_{\bt}, \by-\bg\rangle_{\bt}
  \geq \|\by-\hat{y}_{\bt}\|_{\bt}^{2}
  - \|\by-\hat{y}_{\bt}\|_{\bt}\|\by-\bg\|_{\bt}.
\end{equation}

By the AM-GM inequality, we know that
$\|\by-\hat{y}_{\bt}\|_{\bt}\|\by-\bg\|_{\bt}\leq
\frac{1}{2}(\|\by-\hat{y}_{\bt}\|_{\bt}^{2} +\|\by-\bg\|_{\bt}^{2})$.
Plugging this
into~\eqref{eq:innerprod}, we get
$$
  \langle \by-\hat{y}_{\bt}, \bg \rangle_{\bt}
  \geq
  \frac{1}{2}(\|\by-\hat{y}_{\bt}\|_{\bt}^{2} - \|\by-\bg\|_{\bt}^{2}).
$$

Now, squaring both sides and using the assumption that $R_{K-1}(\bt)>0$, we have
\begin{equation}
  |\langle \by-\hat{y}_{\bt}, \bg \rangle_{\bt} |^{2}
  \geq
  \frac{1}{4}(\|\by-\hat{y}_{\bt}\|_{\bt}^{2} - \|\by-\bg\|_{\bt}^{2})^{2}
  = \frac{1}{4}R^2_{K-1}(\bt).
\end{equation}

Now we can put the bounds on the numerator and denominator together to get the
desired result:
\begin{equation}
  \max_{(b, \ba) \in \mathbb{R}^{1+p}}\widehat{\Delta}(b, \ba, \bt)
  \geq
  \frac{w(\bt)R^2_{K-1}(\bt)}
  {\|g\|^2_{\L_1(\bt)}}.
\end{equation}
\end{proof}

\subsection{Recursive Inequality}
Here we provide a solution to a simple recursive inequality.
\begin{lemma}
~\label{lem:rec}
 Let $\{a_k\}$ be a decreasing sequence of numbers and
 $\{b_k\}$ be a positive sequence numbers satisfying the following recursive
 expression:
 \begin{equation*}
   a_k
   \leq
   a_{k-1}(1-b_k a_{k-1}), \quad k = 1, 2, \dots, K.
 \end{equation*}
 Then,
$$
   a_K
   \leq
   \frac{1}{\sum_{k=1}^K b_k}, \quad K = 1, 2, \dots
$$
\end{lemma}

\begin{proof}[Proof of Lemma~\ref{lem:rec}]
We may assume without loss of generality that $ a_{K-1} > 0 $;
otherwise the result holds trivially since
$ a_K \leq a_{K-1} \leq 0 \leq \frac{1}{\sum_{k=1}^K b_k} $.
For $K=1$,
\begin{equation*}
  a_1
  \leq
  a_0(1-b_1a_0)
  \leq
  \frac{1}{4b_1}
  < \frac{1}{b_1}.
\end{equation*}

For $ K > 1 $, assume
$a_{K-1}\leq \frac{1}{\sum_{k=1}^{K-1} b_k}$.
Then, either
$a_{K-1}\leq \frac{1}{\sum_{k=1}^{K} b_k}$,
in which case we are done since $a_K \leq a_{K-1}$, or,
$a_{K-1}\geq \frac{1}{\sum_{k=1}^{K} b_k}$, in which case,
\[
  a_K
    \leq
    a_{K-1}(1-b_{K}a_{K-1})
    \leq
    \frac{1}{\sum_{k=1}^{K-1}b_k}
    \Bigg( 1- \frac{b_K}{\sum_{k=1}^{K} b_k} \Bigg)
    =
    \frac{1}{\sum_{k=1}^{K} b_k}.
    \qedhere
\]
\end{proof}

\subsection{Sedrakyan's Inequality}
For completeness, we reproduce Sedrakyan's inequality
\citep{sedrakyan1997applications} in its generalized form below.
\begin{lemma}[Sedrakyan's inequality \citep{sedrakyan1997applications}]
~\label{lmm:sedrakyan}
  Let $ U $ and $ V $ be two non-negative random variables with $ V > 0 $ almost
  surely. Then
  \[
  \E\Bigg[\frac{U}{V}\Bigg] \geq \frac{\big(\E\big[\sqrt{U}\big]\big)^2}{\E[V]}.
  \]
\end{lemma}

\begin{proof}[Proof of Lemma~\ref{lmm:sedrakyan}]
By the Cauchy-Schwarz inequality,
\[
\E\big[\sqrt{U}\big]
= \E\Bigg[\sqrt{\frac{U}{V}}\sqrt{V}\Bigg]
\leq \sqrt{\E\Bigg[\frac{U}{V}\Bigg]}\sqrt{\E[V]}.
\]
Rearranging the above inequality gives the desired result.
\end{proof}

\bibliographystyle{apalike}
\bibliography{tree.bib}

\begin{thebibliography}{}

\bibitem[Barron, 1993]{barron1993universal}
Barron, A.~R. (1993).
\newblock Universal approximation bounds for superpositions of a sigmoidal
  function.
\newblock {\em IEEE Trans. Inform. Theory}, 39(3):930--945.

\bibitem[Barron, 1994]{barron1994approximation}
Barron, A.~R. (1994).
\newblock Approximation and estimation bounds for artificial neural networks.
\newblock {\em Machine Learning}, 14(1):115--133.

\bibitem[Barron et~al., 2008]{barron2008approximation}
Barron, A.~R., Cohen, A., Dahmen, W., and DeVore, R.~A. (2008).
\newblock Approximation and learning by greedy algorithms.
\newblock {\em Annals of Statistics}, 36(1):64--94.

\bibitem[Bennett, 1994]{bennett1994global}
Bennett, K.~P. (1994).
\newblock Global tree optimization: A non-greedy decision tree algorithm.
\newblock {\em Journal of Computing Science and Statistics}, pages 156--156.

\bibitem[Bertsimas and Dunn, 2017]{bertsimas2017optimal}
Bertsimas, D. and Dunn, J. (2017).
\newblock Optimal classification trees.
\newblock {\em Machine Learning}, 106(7):1039--1082.

\bibitem[Bertsimas and Dunn, 2019]{bertsimas2019machine}
Bertsimas, D. and Dunn, J. (2019).
\newblock {\em Machine learning under a modern optimization lens}.
\newblock Dynamic Ideas LLC.

\bibitem[Bertsimas et~al., 2021]{bertsimas2021near}
Bertsimas, D., Dunn, J., and Wang, Y. (2021).
\newblock Near-optimal nonlinear regression trees.
\newblock {\em Operations Research Letters}, 49(2):201--206.

\bibitem[Bertsimas et~al., 2018]{bertsimas2018optimal}
Bertsimas, D., Mazumder, R., and Sobiesk, M. (2018).
\newblock Optimal classification and regression trees with hyperplanes are as
  powerful as classification and regression neural networks.
\newblock {\em Unpublished manuscript}.

\bibitem[Bertsimas and Stellato, 2021]{bertsimas2021voice}
Bertsimas, D. and Stellato, B. (2021).
\newblock The voice of optimization.
\newblock {\em Machine Learning}, 110(2):249--277.

\bibitem[Breiman, 2001]{breiman2001random}
Breiman, L. (2001).
\newblock Random forests.
\newblock {\em Machine Learning}, 45(1):5--32.

\bibitem[Breiman et~al., 1984]{breiman1984cart}
Breiman, L., Friedman, J., Olshen, R., and Stone, C. (1984).
\newblock {\em Classification and Regression Trees}.
\newblock Belmont, Calif.: Wadsworth International Group, c1984.

\bibitem[Brodley and Utgoff, 1995]{brodley1995multivariate}
Brodley, C.~E. and Utgoff, P.~E. (1995).
\newblock Multivariate decision trees.
\newblock {\em Machine Learning}, 19(1):45--77.

\bibitem[Buciluundefined et~al., 2006]{Buciluundefined2006}
Buciluundefined, C., Caruana, R., and Niculescu-Mizil, A. (2006).
\newblock Model compression.
\newblock In {\em Proceedings of the 12th ACM SIGKDD International Conference
  on Knowledge Discovery and Data Mining}, KDD '06, pages 535--541, New York,
  NY, USA. Association for Computing Machinery.

\bibitem[Cattaneo et~al., 2020]{cattaneo2020large}
Cattaneo, M.~D., Farrell, M.~H., and Feng, Y. (2020).
\newblock Large sample properties of partitioning-based series estimators.
\newblock {\em The Annals of Statistics}, 48(3):1718--1741.

\bibitem[Chi et~al., 2022]{chi2020asymptotic}
Chi, C.-M., Vossler, P., Fan, Y., and Lv, J. (2022).
\newblock Asymptotic properties of high-dimensional random forests.
\newblock {\em The Annals of Statistics}, 50(6):3415 -- 3438.

\bibitem[DeVore et~al., 2023]{devore2023weighted}
DeVore, R., Nowak, R.~D., Parhi, R., and Siegel, J.~W. (2023).
\newblock Weighted variation spaces and approximation by shallow relu networks.
\newblock {\em arXiv preprint arXiv:2307.15772}.

\bibitem[Dunn, 2018]{dunn2018optimal}
Dunn, J.~W. (2018).
\newblock {\em Optimal trees for prediction and prescription}.
\newblock PhD thesis, Massachusetts Institute of Technology.

\bibitem[Durrett, 2019]{durrett_2019}
Durrett, R. (2019).
\newblock {\em Probability: Theory and Examples}.
\newblock Cambridge Series in Statistical and Probabilistic Mathematics.
  Cambridge University Press, 5 edition.

\bibitem[Frosst and Hinton, 2017]{frosst2017distilling}
Frosst, N. and Hinton, G. (2017).
\newblock Distilling a neural network into a soft decision tree.
\newblock {\em arXiv preprint arXiv:1711.09784}.

\bibitem[Ghosh et~al., 2021]{ghosh2021efficient}
Ghosh, P., Azam, S., Jonkman, M., Karim, A., Shamrat, F. M. J.~M., Ignatious,
  E., Shultana, S., Beeravolu, A.~R., and De~Boer, F. (2021).
\newblock Efficient prediction of cardiovascular disease using machine learning
  algorithms with relief and lasso feature selection techniques.
\newblock {\em IEEE Access}, 9:19304--19326.

\bibitem[Gy{\"o}rfi et~al., 2002]{gyorfi2002distribution}
Gy{\"o}rfi, L., Kohler, M., Krzy{\.z}ak, A., and Walk, H. (2002).
\newblock {\em A Distribution-Free Theory of Nonparametric Regression},
  volume~1.
\newblock Springer.

\bibitem[Hastie et~al., 2009]{Hastie-Tibshirani-Friedman2009_book}
Hastie, T., Tibshirani, R., and Friedman, J. (2009).
\newblock {\em The Elements of Statistical Learning}.
\newblock Springer Series in Statistics. Springer, New York.

\bibitem[Heath et~al., 1993]{Heath93inductionof}
Heath, D., Kasif, S., and Salzberg, S. (1993).
\newblock Induction of oblique decision trees.
\newblock {\em Journal of Artificial Intelligence Research}, 2(2):1--32.

\bibitem[Hornik, 1991]{hornik1991approximation}
Hornik, K. (1991).
\newblock Approximation capabilities of multilayer feedforward networks.
\newblock {\em Neural Networks}, 4(2):251--257.

\bibitem[Huang, 2003]{huang2003local}
Huang, J.~Z. (2003).
\newblock Local asymptotics for polynomial spline regression.
\newblock {\em The Annals of Statistics}, 31(5):1600--1635.

\bibitem[H{\"u}llermeier et~al., 2021]{hullermeier2021automated}
H{\"u}llermeier, E., Mohr, F., Tornede, A., and Wever, M. (2021).
\newblock Automated machine learning, bounded rationality, and rational
  metareasoning.
\newblock {\em arXiv preprint arXiv:2109.04744}.

\bibitem[Klusowski, 2020]{klusowski2020sparse}
Klusowski, J.~M. (2020).
\newblock {Sparse Learning with CART}.
\newblock In Larochelle, H., Ranzato, M., Hadsell, R., Balcan, M.~F., and Lin,
  H., editors, {\em Advances in Neural Information Processing Systems},
  volume~33, pages 11612--11622. Curran Associates, Inc.

\bibitem[Klusowski and Tian, 2022]{klusowski2022large}
Klusowski, J.~M. and Tian, P. (2022).
\newblock Large scale prediction with decision trees.
\newblock {\em Journal of the American Statistical Association}.

\bibitem[Lee and Jaakkola, 2020]{Lee2020Oblique}
Lee, G.-H. and Jaakkola, T.~S. (2020).
\newblock Oblique decision trees from derivatives of relu networks.
\newblock In {\em International Conference on Learning Representations}.

\bibitem[Li et~al., 2003]{li2003multivariate}
Li, X.-B., Sweigart, J., Teng, J., Donohue, J., Thombs, L., and Wang, S.
  (2003).
\newblock Multivariate decision trees using linear discriminants and tabu
  search.
\newblock {\em IEEE Transactions on Systems, Man, and Cybernetics - Part A:
  Systems and Humans}, 33(2):194--205.

\bibitem[Loh and Shih, 1997]{loh1997split}
Loh, W.-Y. and Shih, Y.-S. (1997).
\newblock Split selection methods for classification trees.
\newblock {\em Statistica Sinica}, 7(4):815--840.

\bibitem[López-Chau et~al., 2013]{lopez2013fisher}
López-Chau, A., Cervantes, J., López-García, L., and Lamont, F.~G. (2013).
\newblock Fisher’s decision tree.
\newblock {\em Expert Systems with Applications}, 40(16):6283--6291.

\bibitem[Menze et~al., 2011]{menze2011oblique}
Menze, B.~H., Kelm, B.~M., Splitthoff, D.~N., Koethe, U., and Hamprecht, F.~A.
  (2011).
\newblock On oblique random forests.
\newblock In {\em Joint European Conference on Machine Learning and Knowledge
  Discovery in Databases}, pages 453--469. Springer.

\bibitem[Mingers, 1989]{mingers1989empirical}
Mingers, J. (1989).
\newblock An empirical comparison of pruning methods for decision tree
  induction.
\newblock {\em Machine learning}, 4(2):227--243.

\bibitem[Murdoch et~al., 2019]{murdoch2019definitions}
Murdoch, W.~J., Singh, C., Kumbier, K., Abbasi-Asl, R., and Yu, B. (2019).
\newblock Definitions, methods, and applications in interpretable machine
  learning.
\newblock {\em Proceedings of the National Academy of Sciences},
  116(44):22071--22080.

\bibitem[Murthy et~al., 1994]{murthy1994system}
Murthy, S.~K., Kasif, S., and Salzberg, S. (1994).
\newblock A system for induction of oblique decision trees.
\newblock {\em Journal of Artificial Intelligence Research}, 2(1):1--32.

\bibitem[Nobel, 1996]{nobel1996histogram}
Nobel, A. (1996).
\newblock {Histogram regression estimation using data-dependent partitions}.
\newblock {\em The Annals of Statistics}, 24(3):1084 -- 1105.

\bibitem[Parhi and Nowak, 2023]{parhi2023deep}
Parhi, R. and Nowak, R.~D. (2023).
\newblock Deep learning meets sparse regularization: A signal processing
  perspective.
\newblock {\em arXiv preprint arXiv:2301.09554}.

\bibitem[Quinlan, 1993]{quinlan1993programs}
Quinlan, J.~R. (1993).
\newblock C4.5, programs for machine learning.
\newblock {\em In Proc. of 10th International Conference on Machine Learning},
  pages 252--259.

\bibitem[Raymaekers et~al., 2023]{raymaekers2023fast}
Raymaekers, J., Rousseeuw, P.~J., Verdonck, T., and Yao, R. (2023).
\newblock Fast linear model trees by pilot.
\newblock {\em arXiv preprint arXiv:2302.03931}.

\bibitem[Rodriguez et~al., 2006]{rodriguez2006rotation}
Rodriguez, J., Kuncheva, L., and Alonso, C. (2006).
\newblock Rotation forest: A new classifier ensemble method.
\newblock {\em IEEE Transactions on Pattern Analysis and Machine Intelligence},
  28(10):1619--1630.

\bibitem[Rudin, 2019]{rudin2019interpretable}
Rudin, C. (2019).
\newblock Stop explaining black box machine learning models for high stakes
  decisions and use interpretable models instead.
\newblock {\em Nature Machine Intelligence}, 1:206–215.

\bibitem[Scornet et~al., 2015]{scornet2015consistency}
Scornet, E., Biau, G., and Vert, J.-P. (2015).
\newblock {Consistency of random forests}.
\newblock {\em The Annals of Statistics}, 43(4):1716 -- 1741.

\bibitem[Sedrakyan, 1997]{sedrakyan1997applications}
Sedrakyan, N. (1997).
\newblock About the applications of one useful inequality.
\newblock {\em Kvant Journal}, 97(2):42--44.

\bibitem[Syrgkanis and Zampetakis, 2020]{syrgkanis2020estimation}
Syrgkanis, V. and Zampetakis, M. (2020).
\newblock Estimation and inference with trees and forests in high dimensions.
\newblock In Abernethy, J. and Agarwal, S., editors, {\em Proceedings of Thirty
  Third Conference on Learning Theory}, volume 125 of {\em Proceedings of
  Machine Learning Research}, pages 3453--3454. PMLR.

\bibitem[Tomita et~al., 2020]{tomita2020sparse}
Tomita, T.~M., Browne, J., Shen, C., Chung, J., Patsolic, J.~L., Falk, B.,
  Priebe, C.~E., Yim, J., Burns, R., Maggioni, M., and Vogelstein, J.~T.
  (2020).
\newblock Sparse projection oblique randomer forests.
\newblock {\em Journal of Machine Learning Research}, 21(104):1--39.

\bibitem[Wager and Athey, 2018]{wager2018estimation}
Wager, S. and Athey, S. (2018).
\newblock Estimation and inference of heterogeneous treatment effects using
  random forests.
\newblock {\em Journal of the American Statistical Association},
  113(523):1228--1242.

\bibitem[Yang et~al., 2018]{yang2018deep}
Yang, Y., Morillo, I.~G., and Hospedales, T.~M. (2018).
\newblock Deep neural decision trees.
\newblock In {\em ICML Workshop on Human Interpretability in Machine Learning
  (WHI)}.

\bibitem[Zhan et~al., 2023]{zhan2023consistency}
Zhan, H., Liu, Y., and Xia, Y. (2023).
\newblock Consistency of the oblique decision tree and its random forest.
\newblock {\em arXiv preprint arXiv:2211.12653}.

\bibitem[Zhang, 2003]{zhang2003sequential}
Zhang, T. (2003).
\newblock Sequential greedy approximation for certain convex optimization
  problems.
\newblock {\em IEEE Transactions on Information Theory}, 49(3):682--691.

\bibitem[Zhu et~al., 2020]{zhu2020scalable}
Zhu, H., Murali, P., Phan, D., Nguyen, L., and Kalagnanam, J. (2020).
\newblock A scalable mip-based method for learning optimal multivariate
  decision trees.
\newblock In Larochelle, H., Ranzato, M., Hadsell, R., Balcan, M., and Lin, H.,
  editors, {\em Advances in Neural Information Processing Systems}, volume~33,
  pages 1771--1781. Curran Associates, Inc.

\end{thebibliography}
\end{document}